\documentclass[pdflatex,sn-mathphys-num]{sn-jnl}

\usepackage{graphicx}%
\usepackage{multirow}%
\usepackage{amsmath,amssymb,amsfonts}%
\usepackage{amsthm}%
\usepackage{mathrsfs}%
\usepackage[title]{appendix}%
\usepackage{xcolor}%
\usepackage{textcomp}%
\usepackage{manyfoot}%
\usepackage{booktabs}%
\usepackage{longtable}%
\usepackage{algorithm}%
\usepackage{algorithmicx}%
\usepackage{algpseudocode}%
\usepackage{listings}%
\usepackage{subcaption}
\usepackage{tikz}
\usepackage{url}
\usepackage{comment}

\usepackage{todonotes}
\newcommand{\Real}{\mathbb{R}}
\newcommand{\To}{\rightarrow}
\numberwithin{equation}{section}
\theoremstyle{thmstyleone}%
\newtheorem{theorem}{Theorem}
\theoremstyle{thmstyletwo}%
\newtheorem{remark}{Remark}%

\theoremstyle{thmstylethree}%
\newtheorem{lemma}[theorem]{Lemma}

\begin{document}

\title[Global Optimization Out-of-the-Box]{Out-of-the-Box Global Optimization for Packing Problems: New Models and Improved Solutions}


\author[1,2]{\fnm{Timo} \sur{Berthold}}\email{timoberthold@fico.com}

\author[3]{\fnm{Dominik} \sur{Kamp}}\email{kamp@zib.de}

\author[3]{\fnm{Gioni} \sur{Mexi}}\email{mexi@zib.de}

\author[2,3]{\fnm{Sebastian} \sur{Pokutta}}\email{pokutta@zib.de}

\author[4]{\fnm{Imre} \sur{P\'olik}}\email{imre@polik.net}

\affil[1]{\orgname{Fair Isaac Deutschland GmbH}, \orgaddress{\country{Germany}}}

\affil[2]{\orgdiv{Institut f\"ur Mathematik}, \orgname{Technische Universit\"at Berlin}, \orgaddress{\country{Germany}}}

\affil[3]{\orgname{Zuse Institute Berlin}, \orgaddress{\country{Germany}}}

\affil[4]{\orgname{Fair Isaac Corporation}, \orgaddress{\country{USA}}}

\abstract{
Recent discoveries driven by Large Language Models (LLMs) have renewed interest in geometric packing problems.
In this paper, we study several classes of such packing problems through the lens of modern global nonlinear optimization.
Starting from direct nonlinear formulations, we consider packing circles in squares and fixed-perimeter rectangles, packing circles into minimum-area ellipses, packing regular polygons into regular polygons, and packing Platonic solids into Platonic solids.
For ellipse packing, we derive a novel containment formulation based on the S-lemma.
For polygon and Platonic-solid packing, we develop compact non-overlap formulations based on the Farkas lemma and study several natural modeling variants computationally.

Using off-the-shelf global optimization solvers, namely FICO Xpress and SCIP, we obtain numerous new incumbent solutions as well as first solutions for previously unstudied variants without any problem-specific solver modifications beyond writing down the models.
Our computational results analyze the impact of various formulation choices.

Beyond the individual packing results, the paper illustrates a broader point.
It provides further evidence that global nonlinear optimization has matured into an increasingly practical model-and-solve technology even for highly nonconvex problems.
}


\maketitle

\section{Introduction}\label{sec:intro}

\subsection{Motivation}

The rapid progress in global optimization technology over the past decade has substantially expanded the range of nonlinear, nonconvex problems that can be solved reliably by general-purpose optimization software.
State-of-the-art academic solvers such as SCIP~\cite{SCIPOptSuite10} and commercial solvers such as FICO\textsuperscript{\textregistered} Xpress~\cite{BelottiBertholdGallyGottwaldPolik2025}, Gurobi or Hexaly combine spatial branch-and-bound, automatic linearization and convexification, sophisticated presolving and propagation, and increasingly powerful primal heuristics~\cite{Berthold2014,berthold2018computational,Berthold_Lodi_Salvagnin_2025}.
As a consequence, classes of nonlinear optimization problems that would only recently have been considered computationally far beyond the reach of generic solvers can now often be solved to proven global optimality, or at least to very high-quality primal solutions accompanied by meaningful dual bounds.

At the same time, recent developments in algorithm design based on Large Language Models (LLMs) drew attention to long-standing geometric and combinatorial problems that admit natural nonlinear optimization formulations.
In particular, Google DeepMind presented the AlphaEvolve framework~\cite{georgiev2025mathematical,novikov2025alphaevolve}, which uses LLM-generated code in an evolutionary search to produce high-quality solutions for an extensive set of mathematical problems, including variants of circle packing and hexagon packing.
These developments raise a natural question: to what extent can state-of-the-art global optimization solvers match, or even surpass, such discovered algorithms when the underlying problems are formulated directly as optimization models and solved with off-the-shelf nonlinear optimization technology?\footnote{Curiously, in a study with OpenEvolve~\cite{openevolve}, a project aimed at replicating the AlphaEvolve project using open-source tools, the LLM converged to using an optimization approach~\cite{circlesSLSQP}, namely a Sequential Least Squares Programming formulation with a local solver implemented in SciPy~\cite{SciPySLSQP,virtanen2020scipy}.} 

Packing problems provide a particularly appealing problem class for studying this question.
They have a long history in mathematics and optimization, with a track record of best-known solutions~\cite{friedmanPacking}, and they feature highly nonconvex feasible regions.
At the same time, they admit relatively intuitive nonlinear formulations, making them well suited for studying the interaction between modeling choices and solver performance.
In this sense, packing problems are not only interesting in their own right, but also serve as a neat showcase for understanding the current capabilities and limitations of modern global nonlinear optimization.

\subsection{Contributions}

In this paper, we revisit several families of geometric packing problems from a ``model-and-solve'' perspective.
The modeling effort differs between the problem classes.
We begin with straightforward nonlinear formulations for packing circles in squares and fixed-perimeter rectangles, using only the textbook definitions of containment and non-overlap and requiring no expert knowledge.
For regular polygon packing, we develop a compact non-overlap formulation based on the Farkas lemma~\cite{Farkas1902}.
Deriving it requires some background in optimization, but should be accessible to anyone with a graduate education in the field, and related ideas exist in the literature~\cite{kallrath2009cutting,stoyan2016quasi,romanova2018packing}.
We straightforwardly extend this modeling idea to the three-dimensional setting of packing Platonic solids into Platonic solids.
To the best of our knowledge, this yields the first systematic optimization study of that broader problem class beyond the special case of packing cubes into cubes.
For packing circles into ellipses, we derive a novel containment model based on the S-lemma~\cite{polik07slemma}.
This step requires the most expertise, and we regard it as the paper's main modeling contribution.
Across all classes, we did not develop any problem-specific solution algorithm or heuristic.
The main requirement is formulating the problem mathematically, i.e., as an optimization model; the solution search is left to the solver.

More specifically, our contributions are threefold:
\begin{enumerate}
\item On the modeling side, we introduce new formulations for ellipse, polygon and Platonic-solid packing, including a practical use of the S-lemma for circle-in-ellipse containment and a Farkas-based separation framework for polygonal and polyhedral non-overlap constraints.
\item On the computational side, we obtain numerous new incumbent solutions, including for the well-studied topics of circle and cube packing, as well as first solutions for variants that appear not to have been studied before.
Notably, all computational results are obtained using off-the-shelf global optimization solvers, without any problem-specific solver modifications or specialized heuristics.
\item On the methodological side, we show that formulation choices matter. For packing polygons and Platonic solids, we compare several natural model variants and observe that redundant strengthening constraints can have a significant positive computational impact, while seemingly reasonable symmetry-breaking constraints may, in fact, degrade performance.
\end{enumerate}
Parts of this paper build on our earlier study in~\cite{berthold2026global}.
In particular, the circle-packing results in Section~\ref{sec:circlepacking} and a subset of the polygon-packing results in Section~\ref{sec:polygonpacking} were already reported there.
The present paper goes substantially beyond that earlier work by broadening the range of packing classes considered, introducing new modeling ideas for ellipses and higher-dimensional polyhedral packing, and adding a systematic empirical study of formulation variants.

The central message of this paper is not restricted to packing.
Rather, our computational study provides further evidence that global nonlinear optimization is increasingly becoming an out-of-the-box technology: transparent mathematical models, combined with modern general-purpose solvers, produce highly competitive results on difficult nonconvex benchmark problems that have recently attracted attention in the context of LLM-driven discovery. See, e.g.,~\cite{sudermannmerx2026} for an excellent recent study on the Heilbronn problem.

\subsection{Organization of the Paper}

The remainder of the paper is organized as follows.
Section~\ref{sec:background} reviews relevant background on packing problems and global nonlinear optimization.
Section~\ref{sec:details} discusses implementation aspects, including solution polishing and the computational setup.
Section~\ref{sec:circlepacking} considers packing circles with variable radii into squares and fixed-perimeter rectangles.
Section~\ref{sec:ellipsepacking} introduces the ellipse-packing model based on the S-lemma.
Section~\ref{sec:polygonpacking} presents the generalized polygon-packing formulation including several modeling variants and compares those computationally.
Section~\ref{sec:platonicpacking} extends this model to Platonic-solid packing in three dimensions. 
The model can be easily extended to higher dimensions, but we restrict our computational analysis and visualizations to the three-dimensional case.
Finally, Section~\ref{thats_it_folks} concludes with lessons learned, limitations, and directions for future work.

\section{Background}\label{sec:background}
This section provides an overview and the necessary background on geometric packing problems and global nonlinear optimization. 
The interaction between these two areas is central to this work, as packing models expose nonconvex structure that is amenable to modern global nonlinear solvers.

\subsection{Packing Problems}\label{subsec:packing}
Packing problems ask for the placement of a set of objects into a container subject to feasibility constraints (typically non-overlap and containment), while optimizing an objective such as minimizing container size, minimizing the number of bins, or maximizing packed volume. 
They have a long history in mathematics. The study of sphere packing dates back at least to Kepler's 1611 conjecture that a face-centered cubic arrangement yields the densest packing of equal spheres in three-dimensional space, a claim that was proven nearly four centuries later by Hales~\cite{Hales2005Kepler}.
The study of packing circles in polygonal regions dates back almost 200 years~\cite{bolyai1832tentamen}.
Polygon packing has been studied at least since the 1950s; see~\cite{FejesToth1964} for an early overview.
Today, the study of packing spans classical geometry (e.g., densest sphere/circle packings) and modern operations research (bin packing, strip packing, nesting).
There are systematic problem typologies clarifying structural variants and modeling assumptions, enabling principled algorithm design and benchmarking; see the widely used classification in~\cite{WascherHaussnerSchumann2007}.

In parallel, the community established benchmark repositories to track progress and to compare methods across heterogeneous variants, including early integrated libraries such as PackLib2~\cite{FeketeVanderVeen2007} for multi-dimensional packing problems and 2DPackLib for two-dimensional packing~\cite{IoriDeLimaMartelloMonaci2022}.
Many packing and distance-constrained geometry models can be written using Euclidean norm expressions (e.g., $\lVert x_i-x_j\rVert_2 \ge r_i+r_j$ for non-overlap), yielding nonconvex nonlinear constraints with substantial symmetry.  EuclidLib~\cite{kuznetsov2024convexification} was introduced to capture such nonconvex, Euclidean-norm-driven optimization instances across various applications, and has become a commonly used benchmarking set for global optimization research. Some of its instances have also been included in MINLPLib~\cite{minlplib}. For a comprehensive overview of Euclidean Distance Geometry, see the survey by Liberti et al.~\cite{Liberti2014EDG}.

For circle packing in particular, deterministic global optimization has long been discussed as a viable route to provable solutions and strong bounds, with comprehensive early computational evidence and application discussion in~\cite{locatelli2002packing,CastilloKampasPinter2008}.
The Eternity puzzle~\cite{Eternity} is a packing problem in which triangle-based nonconvex polygon pieces must be arranged to tile a prescribed board. It drew a lot of public attention because a \pounds 1 million prize was offered (and later awarded to two mathematicians) for a complete solution.
Packing models arise throughout manufacturing and logistics: cutting stock and trim-loss reduction (paper, metal, wood, glass, textiles), palletization and container loading, warehouse storage and automated fulfillment, and layout problems in integrated circuit design.
From an optimization standpoint, these applications frequently impose additional constraints beyond pure non-overlap (e.g., grouping, balancing, stability, incompatibilities), which induce mixed-integer structure even when the geometry is continuous~\cite{WascherHaussnerSchumann2007}.

Recent developments in solving packing problems include the work of Khajavirad~\cite{khajavirad2024circle}, which provides a theoretical comparison of convexification techniques, in particular single-row linear optimization (LP) relaxations, multi-row LP relaxations, and semidefinite optimization (SDP) relaxations, for non-overlap constraints for circles and spheres modeled as nonconvex quadratically constrained quadratic problems (QCQPs), and their implications for bound quality and algorithm design. Google DeepMind introduced the AlphaEvolve framework~\cite{georgiev2025mathematical,novikov2025alphaevolve}, which synthesizes code through LLMs to obtain new best-known solutions for a broad range of mathematical optimization tasks, including circle and hexagon packings as well as minimum-distance point configurations. Shortly after, in work that preceded the present article, we improved some of the AlphaEvolve results by using global nonlinear optimization techniques~\cite{berthold2026global,FicoBlog}.
The field is currently very active, with at least six groups contributing new best-known solutions to the benchmark collections curated and continuously updated at~\cite{friedmanPacking}.
Some of these solutions have not (yet) been published elsewhere.
Where no publication exists, we refer to a solution by the name of its contributors and the year listed at~\cite{friedmanPacking}; where an accompanying publication exists, we cite it instead.

\subsection{Nonlinear Optimization}\label{subsec:nlp}

Global nonlinear optimization concerns optimization problems with nonlinear, often nonconvex objectives and constraints, where the goal is to compute solutions that are provably globally optimal.
Formally, a \emph{nonlinear optimization problem (NLP)} is defined as:
\begin{align}
\label{nlp_def}    \min f(x)&\\    
\notag    g_k(x)&\leq 0,\ k = 1,\dots,m\\
\notag    \ell\leq x&\leq u,
\end{align} where $x\in\mathbb{R}^n$, $f, g_k:\mathbb{R}^{n}\rightarrow\mathbb{R}$ are functions built from a list of elementary functions\footnote{The list usually includes powers, roots, exponentials, logarithms and trigonometric functions. In addition, some solvers offer direct support for the absolute value, minimum, maximum or the sigmoid function.} using the four arithmetic operations and composition.  The variable bounds $\ell,u$ are in ${\bar{\mathbb{R}}}^n := \left(\mathbb{R}\cup\{\pm \infty\}\right)^n$. The situation becomes even more challenging in the presence of integrality constraints, leading to \emph{mixed-integer nonlinear optimization (MINLP)}.

From a complexity-theoretic perspective, NLP is fundamentally hard: continuous nonconvex quadratic optimization is $\mathrm{NP}$-hard~\cite{Sahni1974}, and general nonlinear optimization is undecidable~\cite{Matijasevic1970}.
Algorithmically, modern global solvers extend branch-and-bound and branch-and-cut frameworks from mixed-integer linear programming to nonlinear settings, a central ingredient being the construction of tight convex relaxations.
The seminal work of McCormick introduced convex under- and overestimators for bilinear  functions~\cite{McCormick1976}, forming the foundation of spatial branch-and-bound.
The Reformulation-Linearization Technique (RLT)~\cite{SheraliAdams1990} and semidefinite relaxations~\cite{anstreicher2009semidefinite} represent further systematic convexification approaches.

These methodological advances, as well as the comprehensive work on other branch-and-bound sub-algorithms such as domain propagation~\cite{Couenne,GleixnerBertholdMuelleretal.2016}, branching rules~\cite{Couenne,Vigerske2013,berthold2026learning}, primal heuristics~\cite{Berthold2014,Berthold_Lodi_Salvagnin_2025}, and conflict analysis~\cite{berthold2020conflict}, are reflected in today's solver ecosystem;  see~\cite{BertholdGleixnerHeinzVigerske2012} for a computational study.
Dedicated global solvers such as BARON~\cite{TawarmalaniSahinidis2002,TaSa04} and ANTIGONE~\cite{MisenerFloudas2014}, and open-source frameworks such as SCIP~\cite{Achterberg2009,SCIPOptSuite10}, implement spatial branch-and-bound and use external linear (LP), mixed-integer (MIP), and/or local nonlinear (NLP) solvers.
Those works have paved the way for a second wave of global solvers in which existing commercial MIP solvers have been extended to incorporate global optimization capabilities for general NLPs and MINLPs. This was championed by FICO Xpress in 2022~\cite{BelottiBertholdGallyGottwaldPolik2025}, followed by Gurobi and Hexaly in 2023 and 2024, respectively.

\section{Implementation Details}\label{sec:details}

Each of the following sections will have its own set of computational results. Here we describe the common computational setup. In the paper we will list only objective values and will show a few representative or interesting solutions. There is an online supplement to the paper at \url{https://github.com/DominikKamp/Packing}, including all the models, numerical solutions and visualizations.

\subsection{Solvers}

We used SCIP~10.0.0 and FICO Xpress~9.8.0 to run all of our experiments. These represent the state-of-the-art in open-source and commercial global optimization, respectively. We will not report comparisons between the two solvers, as that is not the point of the study.

\subsection{Multistart vs.\ Global Optimization}

Although for most problems we could just run a global optimization solve, it was often better to first run a local optimization solve started from random starting points, and then input this solution into the global solve. Our typical solves were with a time limit of 10\,000s preceded by a 5\,000s multistart ramp-up, keeping the best solution across all runs. For the formulation comparisons, reduced time limits of 1\,000s preceded by a 500s multistart ramp-up were used. This multistart phase is supported natively by both solvers used in this study.

\subsection{Numerical Thresholds and Solution Polishing}\label{sec:polishing}

Nonlinear optimization solvers produce solutions up to a given relative feasibility tolerance, typically $10^{-6}$, meaning constraints may be violated slightly. 
Also, nonlinear solvers can show a tendency to exploit feasibility tolerances to improve the objective value.
Some of the improvements we have achieved are comparable to this value, so we have to be careful about feasibility.
We therefore require that all reported solutions are feasible in double-precision arithmetic: every constraint, evaluated in IEEE 754 double precision, is satisfied with zero tolerance.
Note that this is a statement about floating-point evaluation and not a proof of feasibility in exact real arithmetic.
For the problems from the AlphaEvolve study, we use the validation code published there.
For all other problems, we use our own checker, available in the online supplement.

To achieve this, we employed a three-level approach.
First, all optimization solves used a tighter feasibility tolerance of $\varepsilon = 10^{-8}$.
Second, we modified the constraints to make them tighter, so that the original constraints stay feasible even if the modified constraints are violated by up to $\varepsilon$.
Third, we slightly adjust the solutions returned by the solvers, using arbitrary-precision arithmetic, to remove the remaining margins such that the double-precision check above passes. 
The exact nature of these changes depends on the problem and will be discussed later.

It is important to emphasize that the only reason we are doing this is to make sure that our claimed improvements are fully correct.

\subsection{Computational Setup}

Experiments were run single-threaded and non-exclusively on a 48-core Intel Xeon Gold 6342 CPU with 1,024\,GB RAM\@.
Models were generated in nl file format using PySCIPOpt and solved with SCIP~10.0.0 and FICO Xpress~9.8.0.
For the best results reported in the online supplement we applied five solver setups: SCIP with tolerance~$10^{-8}$, Xpress seeds~1--3 with tolerance~$10^{-8}$, and Xpress with tolerance~$10^{-6}$.

\section{Packing Circles inside a Square or a Rectangle}\label{sec:circlepacking}

Among circle-packing problems, one of the most studied variants concerns packing a fixed number of unit circles into the smallest possible square; a good survey is given by Peikert~\cite{peikert1992packing}. For most instances of up to twenty circles, exact optimal packings are known and can be derived using algebraic geometry techniques based on polynomial systems and Gr\"obner basis computations; see, e.g.,~\cite{buchberger1976theoretical}.

In this section, we investigate a variant for which it is harder to prove optimality since both the positions and radii of the circles are decision variables. Specifically, we consider packing circles of variable radii into a unit square, or, in a mild relaxation, into a rectangle of perimeter four, while maximizing the sum of their radii.
Best-known solutions are mostly due to Cantrell~\cite{friedman2012circle}, but optimality proofs are only known for some trivial cases.

\subsection{Optimization Model}

A key advantage of mathematical optimization approaches is that the optimization model can be changed easily. This facilitates exploring related problems by simply adding or modifying constraints, as illustrated by the two problem variants below.

Given a positive integer $n$, let $\mathcal{N} = \{1, 2, \ldots, n\}$ denote the set of circles. We define the
following decision variables:
\begin{description}
\item[$(x_i, y_i) \in \mathbb{R}^2$:] coordinates of the center of circle $i \in \mathcal{N}$
\item[$r_i \in \mathbb{R}_+$:] radius of circle $i \in \mathcal{N}$
\item[$\alpha \in \mathbb{R}_+$:] width of the rectangle
\end{description}
We consider the problem of packing circles in a rectangle with fixed perimeter
$P = 4$ and variable side lengths. We introduce a decision variable $\alpha$
representing the width and set the height to $H = \frac{P}{2} - \alpha = 2 - \alpha$.
We can assume without loss of generality that $\alpha \leq 1$ (i.e., $\alpha$ is
the shorter side). The optimization problem can be formulated as:
\begin{subequations}\label{eq:circlesquare}
\begin{alignat}{2}
    \max\  \sum_{i \in \mathcal{N}} r_i &
    \label{eq:circle-obj} && \\[0.3em]
    \text{s.t.}\quad
     r_i \le x_i &\le \alpha - r_i,
    &&\textrm{\quad for all } i \in \mathcal{N}
    \label{eq:circle-bound-x} \\
    r_i \le y_i &\le (2-\alpha) - r_i,
    &&\textrm{\quad for all } i \in \mathcal{N}
    \label{eq:circle-bound-y} \\
     (x_i - x_j)^2 + (y_i - y_j)^2 &\ge (r_i + r_j)^2,
    &&\textrm{\quad for all } i,j \in \mathcal{N},\ i<j
    \label{eq:circle-nonoverlap} \\
     0 \le r_i &\le \tfrac{\alpha}{2},
    &&\textrm{\quad for all } i \in \mathcal{N}
    \label{eq:circle-radius-bound} \\
    0 < \alpha &\le 1
    \label{eq:circle-width} &&
\end{alignat}
\end{subequations}
The objective \eqref{eq:circle-obj} is to maximize the sum of all radii.
Note that this is fundamentally different from maximizing the total area covered by circles,
which would be $\pi \sum_{i \in \mathcal{N}} r_i^2$; the linear objective instead tends
to favor more balanced distributions of circle sizes.
Constraints \eqref{eq:circle-bound-x} and \eqref{eq:circle-bound-y} ensure that
each circle remains entirely within the rectangle: the center coordinates must
maintain a distance of at least $r_i$ from all rectangle boundaries.
Constraints \eqref{eq:circle-nonoverlap} ensure that no two circles overlap by
requiring that the Euclidean distance\footnote{It is advantageous to work with squared distances to avoid square roots in the formulation.} between any pair of circle centers is at
least the sum of their radii. Constraints \eqref{eq:circle-radius-bound}
provide an upper bound on each radius: no circle can have a diameter exceeding
the width of the rectangle $\alpha$ (which is the shorter side by assumption).
Finally, constraint \eqref{eq:circle-width} bounds the width variable.

The aspect ratio determined by $\alpha$ is a decision variable that can be modified to maximize
the sum of radii for a given number of circles. We can trivially change this formulation to packing into a unit square by fixing $\alpha = 1$. This is a crucial property of mathematical optimization modeling: the user needs to change only the model and does not have to worry about whether or how this changes the algorithm. The solvers will take care of that. This contrasts with many heuristic approaches, including LLM-generated ones, in which often a new set of heuristics needs to be developed once the model formulation changes or a new constraint is added.

The formulation has only linear and nonconvex quadratic constraints. For larger values of $n$ the quadratic growth in the number of constraints \eqref{eq:circle-nonoverlap} will become dominant.

\subsection{Dual Bounds}

Unfortunately, model \eqref{eq:circlesquare} is very weak in the sense that the dual bounds the solvers obtain for it are far from the optimal objective value. We can do better by adding a valid inequality to the problem.

Using the arithmetic-quadratic mean inequality and the fact that the sum of the areas of the circles is less than the area of the enclosing rectangle, we have
\begin{equation}\label{eq:areaineq}
\sum_{i \in \mathcal{N}} r_i \leq \sqrt{n}\sqrt{\sum_{i \in \mathcal{N}} r_i^2}\leq \sqrt{n}\sqrt{\frac{\alpha(2-\alpha)}{\pi}}\leq\sqrt{\frac{n}{\pi}}.
\end{equation}
In contrast, the solvers typically start at a dual bound of $\frac{n}2$, corresponding to setting each $r_i$ to $\frac12$, and improve gradually from there. This is why we decided to add the area inequality 
\begin{equation}\label{eq:area}
\sum_{i \in \mathcal{N}} r_i^2\leq\frac{1}{\pi}
\end{equation}
to model \eqref{eq:circlesquare} for all of our computational tests, and added a similar area inequality also to the other problems in this paper. This constraint not only helps to tighten the dual bound, but it also helps in finding good feasible solutions. This is because it is a constraint involving all the radii of the circles, and it is the only constraint dealing with more than two circles in the packing.

\subsection{Computational Results}

First, let us focus on the restricted problem, packing circles into a square. The improving solutions\footnote{We have found possibly even more improving solutions, but the original source~\cite{friedman2012circle} for the best-known packings lists only three digits after the decimal point. In almost all cases, we managed to match the best-known solutions from the literature.} (floored to five decimal digits) are reported in Table~\ref{tab:packcircle} and visualized in Figure~\ref{fig:circlepacking}. It is instructive to compare this objective value to the bound we get from inequality \eqref{eq:areaineq}. For the $n=32$ case this yields an upper bound of $\sqrt{\frac{32}{\pi}}\approx3.191538\dots$, which is within $8.57\%$ of the solution. The relative tightness of this bound comes from two factors. The circles in the solutions have similar radii, so the arithmetic-quadratic mean inequality does not sacrifice much. Further, the total area of the circles is close to 1, so the area constraint is also fairly tight.
\begin{table}[ht]
\caption{Improving solutions for the circle-packing problem}\label{tab:packcircle}
\begin{tabular}{@{}ccrrr@{}}
\toprule
Variant &$n$ & Sum of radii & Previous best & Source\\
\midrule
square &32 & 2.93\textbf{957} & 2.93794 &\cite{georgiev2025mathematical,novikov2025alphaevolve} \\
\midrule
rectangle & 26 & 2.63\textbf{930} & 2.638 & \cite{friedmanPacking} \\
rectangle & 27 & 2.6\textbf{9015} & 2.687 & \cite{friedmanPacking} \\
\botrule
\end{tabular}
\end{table}
\begin{figure}[ht]
\begin{center}
\includegraphics[width=0.27\linewidth]{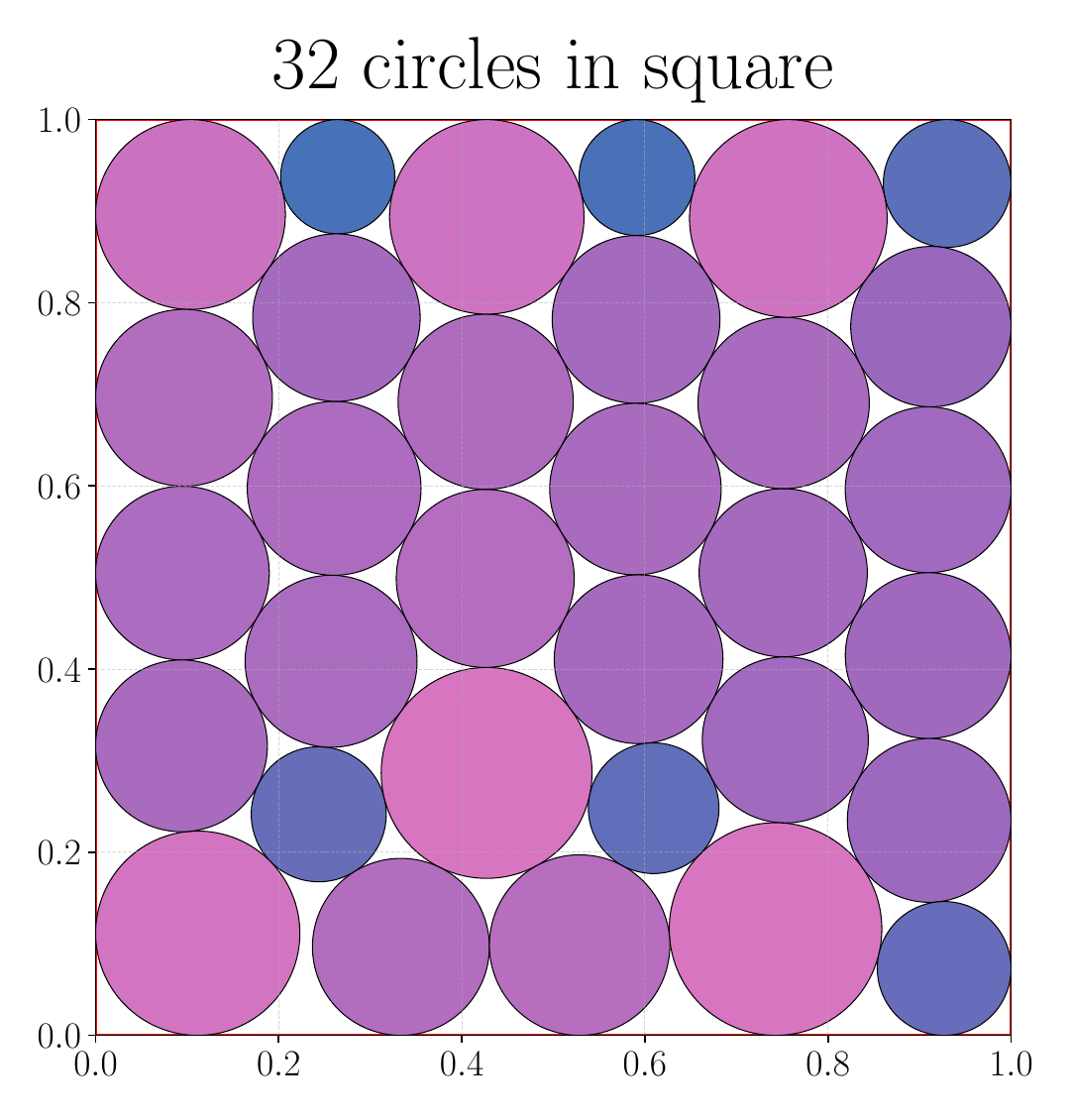}
\includegraphics[width=0.3\linewidth]{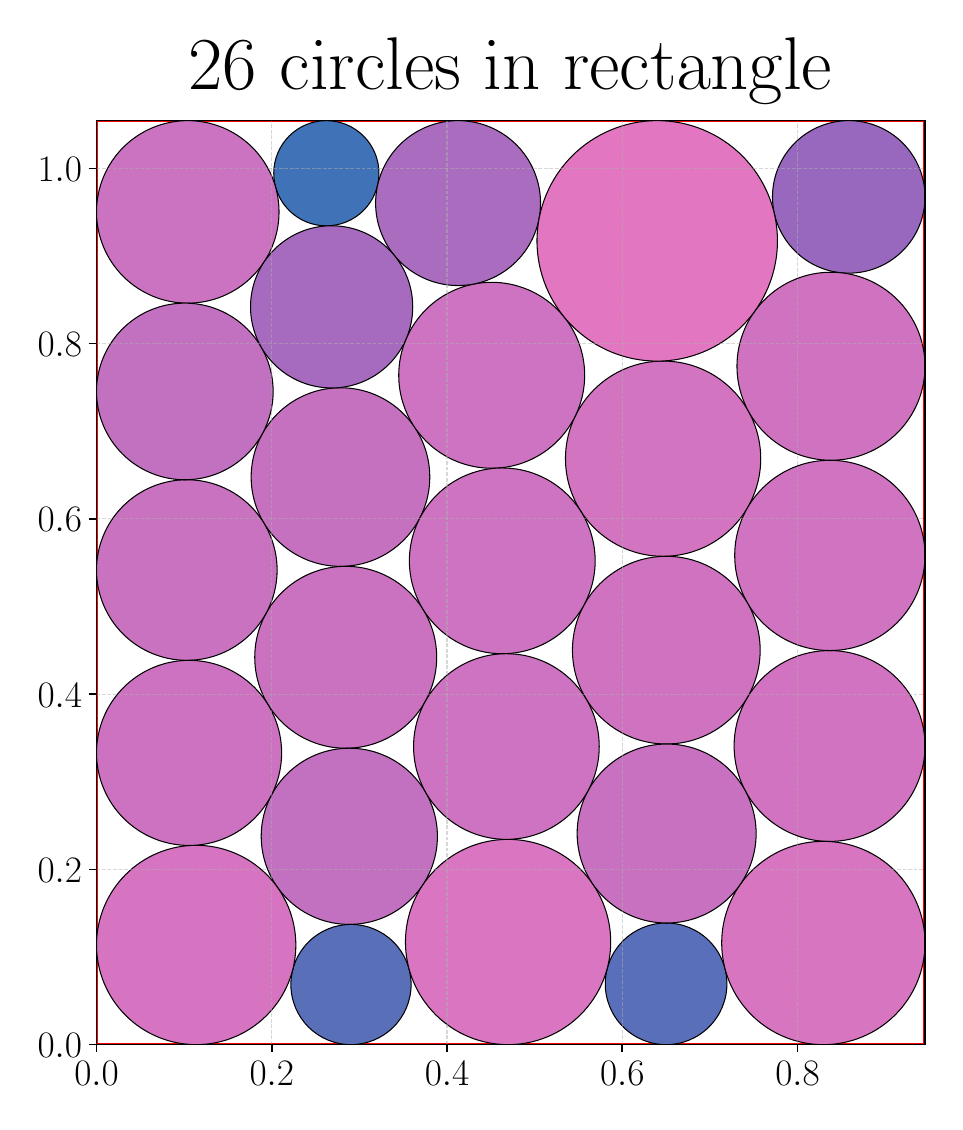}
\includegraphics[width=0.3\linewidth]{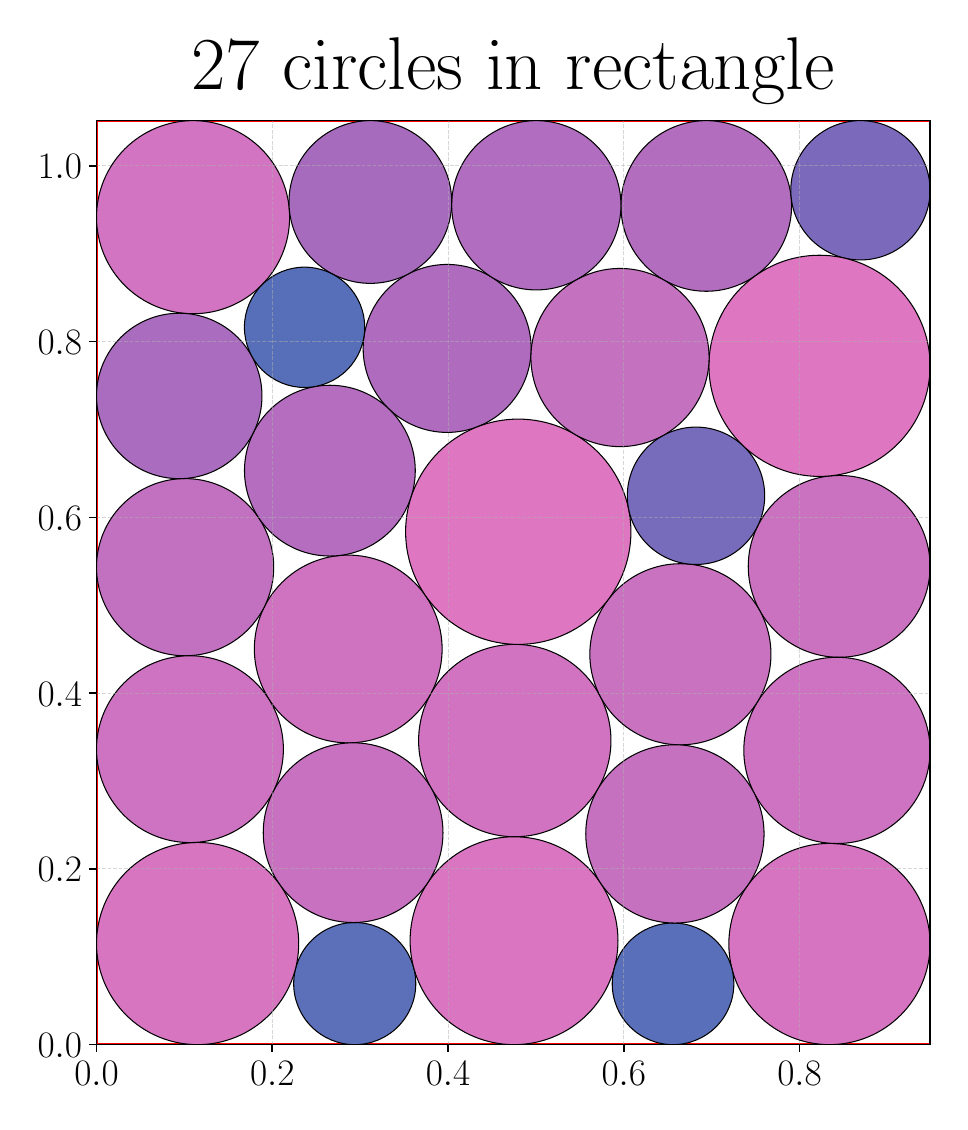}
\end{center}
\caption{Graphical representation of the new solutions for the circle-packing problem. Left: square variant ($n=32$). Middle and right: rectangle variant ($n=26, n=27$). See Table~\ref{tab:packcircle} for details.}\label{fig:circlepacking} 
\end{figure}

Now we can turn to the relaxed version, where instead of a square we are trying to pack into a rectangle of perimeter four. Obviously, all the previous solutions are still feasible for this relaxed problem, but we can do slightly better.
We found the following improving solutions, again
floored to five decimal digits (see Table~\ref{tab:packcircle}). The solutions are visualized in Figure~\ref{fig:circlepacking}. Not surprisingly, the solutions all tend to use a rectangle that is almost square.

To ensure that the solutions are indeed as feasible as we can make them, we scaled the radii in the containment and overlap constraints. The circle centers are kept fixed and the radii are adjusted in two phases. First, the maximum uniform absolute radius increase~$\delta$ that preserves all non-overlap and boundary constraints is computed. Then, each circle's radius is individually increased in order of increasing current radius, each time consuming the slack of its tightest remaining constraint. An absolute safety margin of~$10^{-15}$ is applied to account for numerical rounding errors. These calculations are carried out in arbitrary-precision arithmetic.

We found it quite surprising that new solutions could be found for such a well-studied class of problems with off-the-shelf tools. In fact, it was this class of problems that we first tried from the original AlphaEvolve paper~\cite{novikov2025alphaevolve}. The encouraging initial results led to this much larger project.

\section{Packing Circles into an Ellipse}\label{sec:ellipsepacking}

In this section we present another interesting packing problem, where only very little progress has been made in 20 years. The problem is packing unit circles in an ellipse of minimum area. Known solutions are collected on Erich Friedman's packing website~\cite{friedmanPacking}, the last reported solution being from 2007.

We will first present the optimization model, then see it in action. We will discuss some potential valid inequalities that can be used to strengthen the model.

\subsection{Characterizing the Circle-Ellipse Containment}

The main difficulty of the optimization model is to formalize mathematically that a circle is contained in an ellipse. For this we will use the S-lemma~\cite{yakubovich71}:
\begin{theorem}[S-lemma, Yakubovich, 1971]
Let $f, g:\Real^n \To \Real$ be (not necessarily homogeneous) quadratic functions and suppose that there is an $\bar x\in\Real^n$ such that
$g(\bar x)<0$. Then the following two statements are equivalent.
\begin{enumerate}
\item There is no $x\in\Real^n$ such that
\begin{subequations}\label{eq:slemma_primal}
\begin{eqnarray}
f(x)&<&0,\\
g(x)&\leq&0.
\end{eqnarray}
\end{subequations}
\item There is a nonnegative multiplier $\lambda\geq0$ such that
\begin{equation}
f(x)+\lambda g(x)\geq 0,\,\forall x\in\Real^n.
\end{equation}
\end{enumerate}
\end{theorem}
One could view this result as two quadratic functions seemingly behaving as if they were convex, as far as Lagrange duality is concerned. See~\cite{polik07slemma} for a survey on the S-lemma and its connections to other results.

Now consider an ellipse centered at the origin, with semiaxes $a$ and $b$, and a circle of radius $r$ centered at $(x_0, y_0)$. The circle is contained in the ellipse if and only if
\begin{subequations}
\begin{align}
\notag \nexists x, y&\\
\left(x-x_0\right)^2 + \left(y-y_0\right)^2 &\leq r^2\\
b^2x^2+a^2y^2&>a^2b^2,
\end{align}
\end{subequations}
where the second inequality is an alternative form of the ellipse equation, avoiding the divisions. This can be written equivalently using the S-lemma:
\begin{align}
\label{eq:Snonneg}\exists t \geq 0: \forall x, y\in \mathbb{R}: 
 a^2b^2 - b^2x^2  - a^2y^2  + t\left(\left(x-x_0\right)^2 + \left(y-y_0\right)^2 - r^2\right) \geq0.
\end{align}
This is a non-homogeneous quadratic expression without cross terms. The following lemma summarizes when this expression is nonnegative:
\begin{lemma}
The quadratic expression $Ax^2 + 2Bx + Cy^2 + 2Dy + E$ is nonnegative for all $x,y$ if and only if
\begin{subequations}
\begin{align}
A&\geq 0 \\
C&\geq 0\\
E&\geq 0 \\
AE-B^2&\geq 0\\
CE-D^2&\geq 0\\
AEC-AD^2-CB^2&\geq0.
\end{align}
\end{subequations}
\end{lemma}
\begin{proof}
First, a standard homogenization argument shows that $Ax^2 + 2Bx + Cy^2 + 2Dy + E$ is nonnegative for all $x$ and $y$ if and only if the homogeneous expression $Ax^2 + 2Bxu + Cy^2 + 2Dyu + Eu^2$ is nonnegative for all $x$, $y$ and $u$. Since this is a homogeneous quadratic form, its nonnegativity is equivalent to the following matrix being positive semidefinite:
\begin{equation}
\begin{bmatrix}
A &B &0\\
B &E &D\\
0 &D &C
\end{bmatrix}.
\end{equation}
This, on the other hand, is easy to characterize by setting all principal minors to be nonnegative, yielding exactly the six inequalities\footnote{The nonnegativity of $AC$ follows from that of $A$ and $C$ and can thus be dropped.} in the lemma. Notice that the formulas are greatly simplified because of the 0 appearing in the matrix, due to the lack of an $xy$ term.
\end{proof}
We can apply this result to the expression in \eqref{eq:Snonneg} to get a condition in terms of the problem data.
\begin{subequations}
\begin{align}
A&=t-b^2\\
B&=-tx_0\\
C&=t-a^2\\
D&=-ty_0\\
E&=a^2b^2+t\left(x_0^2+y_0^2-r^2\right)
\end{align}
\end{subequations}
Now we are ready to write the complete optimization problem. We place the ellipse centered at the origin, with semiaxes $a$ and $b$, which are the main decision variables. The $n$ circles are located at $(x_i, y_i), i=1,\dots,n$. For each circle $i$ we will introduce a multiplier $t_i$, which will certify that circle $i$ is contained in the ellipse. The model is
\begin{subequations}
\begin{alignat}{2}
    \min\  a \cdot b \cdot\pi &
    \label{eq:ellipse-obj} && \\[0.3em]
    \text{s.t.}\quad
     (x_i-x_j)^2 + (y_i-y_j)^2 &\geq 4,
    &&\textrm{\quad for all } i < j
    \label{eq:circdisj} \\
     t_i - a^2 &\geq 0,
    &&\textrm{\quad for all } i
    \label{eq:circellipse1} \\
     t_i - b^2 &\geq 0,
    &&\textrm{\quad for all } i
    \label{eq:circellipse0} \\
     a^2 b^2 + t_i (x_i^2 + y_i^2 - 1) &\geq 0,
    &&\textrm{\quad for all } i \\
     (t_i - a^2)(a^2 b^2 + t_i (x_i^2 + y_i^2 - 1)) - t_i^2 y_i^2 &\geq 0,
    &&\textrm{\quad for all } i \\
     (t_i - b^2)(a^2 b^2 + t_i (x_i^2 + y_i^2 - 1)) - t_i^2 x_i^2 &\geq 0,
    &&\textrm{\quad for all } i \\
     (t_i - a^2)(t_i - b^2)(a^2 b^2 + t_i (x_i^2 + y_i^2 - 1)) \nonumber & \\
     {} - t_i^2(x_i^2(t_i - a^2) + y_i^2(t_i - b^2)) &\geq 0,
    &&\textrm{\quad for all } i
    \label{eq:circellipse2} \\
     t_i &\geq 1,
    &&\textrm{\quad for all } i \\
     a, b &\geq 1
    \label{eq:ellipcont} && \\
     a &\geq b
    \label{eq:ellipsym}
\end{alignat}
\end{subequations}
where all indices $i$ and $j$ are in $1,\dots,n$. 

Constraints \eqref{eq:circdisj} ensure that the circles do not overlap. Constraints \eqref{eq:circellipse1}--\eqref{eq:circellipse2} guarantee that the circles are all contained in the ellipse. Constraint \eqref{eq:ellipcont} is a slight strengthening of the nonnegativity of the semiaxes, to guarantee that at least a single circle can be contained in the ellipse. 
Constraint \eqref{eq:ellipsym} ensures that the ellipse has its longer semiaxis along the horizontal axis, which we can assume without loss of generality.

Interestingly, the variables $x_i$ and $y_i$ do not have a finite bound: the optimal ellipse can be elongated. This is also why $a$ and $b$ do not have an upper bound a priori. However, once a feasible solution with objective value $z$ has been found, its objective value can be used to derive an upper bound on $a$, since we can then restrict the solution space to those ellipses that have an area smaller than $z$. In particular, this means that $1\leq b\leq  a\leq \frac{z}{\pi}$. These then imply finite bounds on $x_i$ and $y_i$.

Overall, the model has polynomial degree 8 (coming from the $a^4b^4$ term in constraint \eqref{eq:circellipse2}), which is a rare phenomenon in practice.

\subsection{Strengthening}

First of all we can drop constraint \eqref{eq:circellipse0}, because it is implied by $a\geq b$. Further, we can add a constraint that each circle center is inside the ellipse:
\begin{equation}
b^2x_i^2+a^2y_i^2\leq a^2b^2.
\end{equation}
We can strengthen this further by adding some simple bounds on the centers:
\begin{subequations}
\begin{align}
      x_i &\leq  a - 1\\
      x_i &\geq -a + 1\\
      y_i &\leq   b - 1\\
      y_i &\geq -b + 1
\end{align}
\end{subequations}
These are simple enough not to slow down the optimization. In addition, we could add similar constraints for any point along the circles, such as $\left(x_i\pm1, y_i\right)$ and $\left(x_i, y_i\pm1\right)$. This would be four extra quadratic constraints per circle, which could be too much. Therefore, we did not add these constraints.

Finally, we can add a constraint about the area of the ellipse: it needs to be at least as large as the sum of the areas of the circles:
\begin{equation}
ab \geq n.
\end{equation}
This is an objective cut, which in general does not help the optimization, but here it is beneficial in making sure that the ellipse is large enough. Using the symmetry-breaking constraint \eqref{eq:ellipsym} we can also add the constraint 
\begin{equation}
a \geq \sqrt{n}.
\end{equation}

\subsection{Symmetry Breaking}

There is a lot of symmetry in the formulation. We have already dealt with an obvious one, that of the containing ellipse, by making it longer along the horizontal axis. There are two other sources of symmetry.

One is the mirror symmetry of the ellipse: this allows us to convert any packing into an equivalent one simply by mirroring the location of the circles. To counter this we need to add a constraint that prefers one solution to a symmetric copy. One example is to constrain the centroid of the circles to fall in a certain quarter of the ellipse:
\begin{subequations}
\begin{align}
\label{eq:symbreakx}\sum_{i=1}^nx_i &\geq 0\\
\sum_{i=1}^ny_i &\geq 0.
\end{align}
\end{subequations}
Unless the centroid of the solution is the origin (or lies on one of the coordinate axes), these two constraints will exclude equivalent symmetric solutions. Unfortunately, optimal configurations tend to have some kind of symmetry, so this technique is not perfect.

Another issue is the symmetry of numbering the circles. With $n$ circles we could have $n!$ different solutions just by renumbering them. Therefore, it is common practice\footnote{This would not be necessary for mixed-integer \emph{linear} problems, where exploiting symmetry has long been part of all commercial and even open-source solvers; see~\cite{pfetsch2019symmetry} for a comprehensive survey. On the other hand, since this kind of symmetry is much less common for nonlinear problems, and detecting it is more complicated, it has not been added to nonlinear solvers.} to add a symmetry-breaking constraint. The most natural one is a canonical ordering of the circles, such as ordering their centers in increasing order of the $x$ coordinate:
\begin{equation}
\label{eq:symsortx} x_i\leq x_{i+1},\quad 1\leq i \leq n-1.
\end{equation}
This is not perfect, because two circles can still end up with an identical $x$ coordinate (see the solution for $n=15$ in Figure~\ref{fig:ellipse_test}), but it is simple and drastically reduces the symmetry in the problem. An added benefit of this approach is that it can be used to reduce the symmetry of the configuration itself. In particular, instead of the inequality \eqref{eq:symbreakx} we can write
\begin{equation}
x_{\left\lceil\frac{n}{2}\right\rceil} \geq 0,
\end{equation}
which, together with the sorting constraint \eqref{eq:symsortx}, ensures that at least half of the points are in the right half of the ellipse. This inequality has far fewer nonzeros than constraint \eqref{eq:symbreakx} and achieves a similar result. Note, however, that this cannot be done at the same time for the $y$ coordinate, as we cannot assume that the points are sorted according to both coordinates.

An alternative solution is to order the points along a generic line:
\begin{equation}
\alpha x_i + \beta y_i\leq \alpha x_{i+1} + \beta y_{i+1}, 1\leq i \leq n-1,
\end{equation}
where $\alpha$ and $\beta$ should be selected in a way to make it very unlikely that any of these inequalities is binding. A couple of random irrational numbers generally work, but could harm the numerics of the problem. For this application the choice of $\alpha=2$ and $\beta=3$ would work. This constraint introduces more nonzeros into the problem, but completely eliminates the symmetry coming from reordering the circles.

\subsection{Computational Results}

All solutions were post-processed and polished using the techniques in Section~\ref{sec:polishing}. In particular, to clean up these solutions we added constraints to exploit the (conjectured) structure of the solution: enforcing the apparent symmetry, rounding some values to the coordinate axes, and setting some constraints to equality (to enforce touching objects). We then re-solved the problem to see if these assumptions were correct and the same solution could still be obtained. In all cases this turned out to be true. A similar but more elaborate version of this scheme, employing Gr\"{o}bner bases to obtain the coordinates in closed form, is used in~\cite{sudermannmerx2026}. In the final version of the solutions all constraints are satisfied to double floating-point precision.

\subsubsection{The New Solutions}
We were able to reproduce all of the known solutions\footnote{We do not know how the existing solutions were found and how the circle-ellipse containment was enforced or verified for them.} on Erich Friedman's packing website~\cite{friedmanPacking}. In addition, we found two new improving solutions. These are different configurations, not just slightly improved versions of the previous solutions. 


\begin{figure}[h]
\begin{center}
\includegraphics[width=0.45\linewidth]{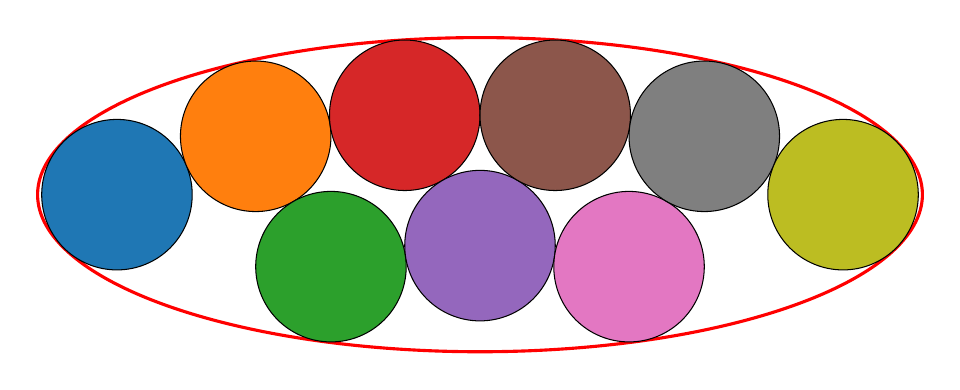}
\includegraphics[width=0.45\linewidth]{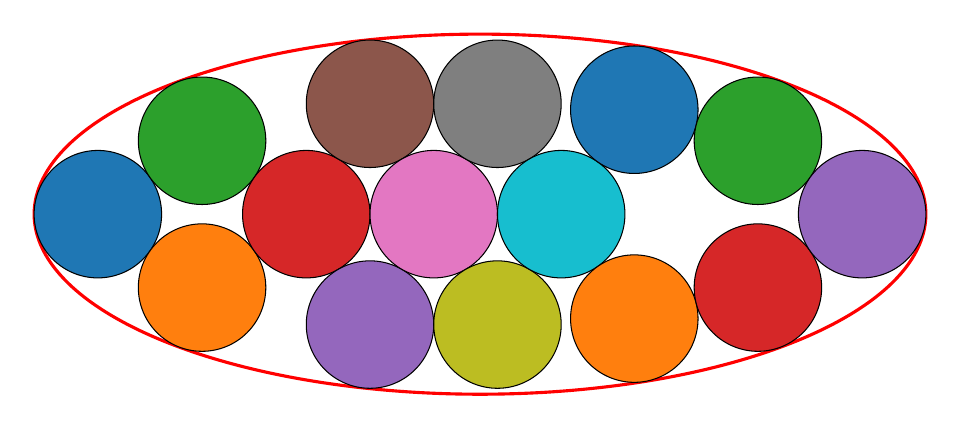}
\end{center}
\caption{New circle-packing arrangements for $n=9$ (left, area $\pi\cdot12.2708\dots$, improving the previous $\pi\cdot 12.403\dots$) and $n=15$ (right, area $\pi\cdot19.76329049\dots$, improving the previous $\pi \cdot19.824\dots$).}
\label{fig:ellipse_test}
\end{figure}


The configuration for $n=9$ in Figure~\ref{fig:ellipse_test} is symmetric about the vertical axis. The semiaxes of the ellipse are $5.8787236$ and $2.087324145$.


The configuration for $n=15$ in Figure~\ref{fig:ellipse_test} is symmetric about the horizontal axis. The semiaxes of the ellipse are $6.995912684$ and $2.824976723$. Interestingly, this solution is not rigid: the two points with $x$ coordinates $0.274\dots$ do not touch the ellipse, so they are free to move. There are thus infinitely many equivalent solutions. The solution we have presented is the one where these circles are touching three other circles.

\section{Packing Polygons} 
\label{sec:polygonpacking}

The problem of packing $n$ regular $m$-gons into a regular $\ell$-gon of minimum circumradius is significantly more
involved than circle packing (Section~\ref{sec:circlepacking}), as each inner element can rotate freely and non-overlap
conditions must account for the specific polygonal geometry of each element.
Optimization models for general polygon packing have been developed using
quasi-phi-functions~\cite{kallrath2009cutting,stoyan2016quasi,romanova2018packing}; we adopt a related approach based
on a variant of the Farkas lemma~\cite{Farkas1902}, which yields a compact formulation applicable to any regular
polygon pair without shape-specific derivations.

\subsection{Optimization Model} 
\label{subsec:polygonmodel}

The problem is to pack $n$ regular $m$-gons with unit circumradius into a regular $\ell$-gon,
minimizing the circumradius $R$ of the outer polygon.
Let $\mathcal{N} = \{1, 2, \ldots, n\}$ denote the set of inner polygons,
$\mathcal{M} = \{0, 1, \ldots, m-1\}$ index both the $m$ vertices and $m$ edges of each inner polygon, and
$\mathcal{L} = \{0, 1, \ldots, \ell-1\}$ denote the $\ell$ edges of the outer polygon. We use
the following geometric constants for regular polygons with unit circumradius:
\begin{description}
\item[$\rho_m = \cos(\pi/m)$:] inradius of inner $m$-gons
\item[$\rho_\ell = \cos(\pi/\ell)$:] inradius of outer $\ell$-gon
\item[$\phi_m = 2\pi/m$:] angular distance between adjacent vertices of $m$-gons
\item[$\phi_\ell = 2\pi/\ell$:] angular distance between adjacent vertices of $\ell$-gon
\item[$\delta_{m,j} = (j + ((m-1) \bmod 2)/2)\phi_m$:] vertex angle for $j \in \mathcal{M}$ of $m$-gons
\item[$R_{\min} = \sqrt{\frac{n \cdot m \cdot \sin(\phi_m)}{\ell \cdot \sin(\phi_\ell)}}$:] lower bound on circumradius from area requirement
\end{description}
We define the following decision variables:
\begin{description}
\item[$R \ge 0$:] circumradius of the scaled $\ell$-gon (to be minimized)
\item[$(x_i, y_i) \in \mathbb{R}^2$:] coordinates of the center of inner $m$-gon $i \in \mathcal{N}$
\item[$\theta_i \in \lbrack0, \phi_m\rbrack$:] rotation angle of inner $m$-gon $i \in \mathcal{N}$
\item[$a_{i,j}, b_{i,j} \in \mathbb{R}$:] oriented inward normal components for edge $j \in \mathcal{M}$ of $m$-gon $i \in \mathcal{N}$
\item[$s_{i,j} \in \mathbb{R}$:] offset for edge $j \in \mathcal{M}$ of $m$-gon $i \in \mathcal{N}$
\item[$\lambda_{i,j,k} \ge 0$:] Farkas multiplier of pair $i, j \in \mathcal{N}$ with $i < j$ and common edge index $k \in \{1, \ldots, 2m\}$
\end{description}
The separation conditions are based on a version of the Farkas lemma in its strict-inequality variant.
More specifically, this is the non-homogeneous version of Gordan's theorem~\cite{gordan1873}, and a special case of Motzkin's transposition theorem, see, e.g.,~\cite{Schrijver1986}:
A system of strict linear inequalities $Ax > b$ is infeasible if and only if there exist nonnegative multipliers $y \geq 0$, $y \neq 0$, with $y^T A = 0$ and $y^T b \geq 0$.

\begin{lemma}
\label{lem:polygon-separation}
Two convex $m$-gons with half-space representations
$\{a_{i,k}x + b_{i,k}y \geq s_{i,k}\}_{k \in \mathcal{M}}$ and $\{a_{j,k}x + b_{j,k}y \geq s_{j,k}\}_{k \in \mathcal{M}}$
have disjoint interiors if and only if there exist nonnegative multipliers $\lambda_1, \ldots, \lambda_{2m}$ summing to
one with
\begin{subequations}
\begin{align}
  \sum_{k\in\mathcal{M}} \lambda_{k+1} a_{i,k} + \sum_{k\in\mathcal{M}} \lambda_{k+m+1} a_{j,k} &= 0,\\
  \sum_{k\in\mathcal{M}} \lambda_{k+1} b_{i,k} + \sum_{k\in\mathcal{M}} \lambda_{k+m+1} b_{j,k} &= 0,\\
  \sum_{k\in\mathcal{M}} \lambda_{k+1} s_{i,k} + \sum_{k\in\mathcal{M}} \lambda_{k+m+1} s_{j,k} &\geq 0,
\end{align}
\end{subequations}
i.e., a canceling normal combination with nonnegative gap.
\end{lemma}

\begin{proof}
The interiors of the two polygons are disjoint if and only if the strict linear system
\begin{subequations}
\begin{align}
  a_{i,k} x + b_{i,k} y &> s_{i,k}, \quad k \in \mathcal{M}, \\
  a_{j,k} x + b_{j,k} y &> s_{j,k}, \quad k \in \mathcal{M}
\end{align}
\end{subequations}
in $(x, y) \in \mathbb{R}^2$ is infeasible.
By the Farkas lemma (and the theorems cited above), this system is infeasible if and only if there exist $\mu_1, \ldots, \mu_{2m} \ge 0$, not all zero, such that
\begin{subequations}
\begin{align}
  \sum_{k \in \mathcal{M}} \mu_k a_{i,k} + \sum_{k \in \mathcal{M}} \mu_{m+k} a_{j,k} &= 0, \\
  \sum_{k \in \mathcal{M}} \mu_k b_{i,k} + \sum_{k \in \mathcal{M}} \mu_{m+k} b_{j,k} &= 0, \\
  \sum_{k \in \mathcal{M}} \mu_k s_{i,k} + \sum_{k \in \mathcal{M}} \mu_{m+k} s_{j,k} &\ge 0.
\end{align}
\end{subequations}
Dividing by $\sum_{p=1}^{2m} \mu_p > 0$ yields multipliers summing to one, satisfying all three conditions of the lemma.
\end{proof}

For each pair $(i, j)$ with $i < j$, the $2m$ Farkas multipliers $\lambda_{i,j,k}$, one per edge of each of the two
$m$-gons, encode this certificate.
Geometrically, the partial normal sums
$\sum_k \lambda_{k+1}(a_{i,k}, b_{i,k})$ and $\sum_k \lambda_{k+m+1}(a_{j,k}, b_{j,k})$
are equal and opposite, defining a normal axis with separating offset
$\sum_k \lambda_{k+1} s_{i,k} + \sum_k \lambda_{k+m+1} s_{j,k}$.
Conversely, given a separating axis with unit normal $(u,v)$ and offset $c$, the multipliers for polygon $i$ are
obtained by solving the LP
\begin{subequations} \label{equ:separationlp}
\begin{align}
  \label{equ:separationlp1}
  \max_{\mu_k \ge 0} \quad & \textstyle\sum_{k\in\mathcal{M}} \mu_k\, s_{i,k} \\
  \label{equ:separationlp2}
  \text{s.t.} \quad & \textstyle\sum_{k\in\mathcal{M}} \mu_k\, (a_{i,k}, b_{i,k}) = -(u,v),
\end{align}
\end{subequations}
whose objective does not exceed $c$ by optimality; solving the same LP for polygon $j$ with the negated direction gives
$\nu_k$, and combining both as $\lambda_{k+1} = \mu_k / S$ and $\lambda_{k+m+1} = \nu_k / S$ with
$S := \sum_k \mu_k + \sum_k \nu_k$ yields feasible Farkas multipliers.
The approach applies to any convex polygon without shape-specific derivations.

\begin{remark}
Since the dual-feasible bases of the LP \eqref{equ:separationlp} are defined by the vertices of the two-dimensional polygon, at most two positive adjacent Farkas multipliers per element in each pair are required to satisfy the separation conditions if possible, independent of the number of polygon vertices.
\end{remark}

The optimization problem can be formulated as:
\allowdisplaybreaks
\begin{subequations}
\begin{alignat}{2}
    \min\  R& \label{eq:polygon-obj} && \\[0.3em]
\quad
     R\rho_\ell + \sin(k\phi_\ell)\bigl(x_i + \sin(\theta_i + \delta_{m,j})\bigr) \nonumber &\\
     + \cos(k\phi_\ell)\bigl(y_i + \cos(\theta_i + \delta_{m,j})\bigr)
      &\ge 0,
    &&\forall i \in \mathcal{N},\ j \in \mathcal{M}, \nonumber \\
    &&&k \in \mathcal{L}
    \label{eq:polygon-containment} \\[0.3em]
     a_{i,j} = \sin(\theta_i + j\phi_m),&
    &&\forall i \in \mathcal{N},\ j \in \mathcal{M}
    \label{eq:polygon-normal-x} \\
     b_{i,j} = \cos(\theta_i + j\phi_m),&
    &&\forall i \in \mathcal{N},\ j \in \mathcal{M}
    \label{eq:polygon-normal-y} \\
     s_{i,j} = a_{i,j}x_i + b_{i,j}y_i - \rho_m,&
    &&\forall i \in \mathcal{N},\ j \in \mathcal{M}
    \label{eq:polygon-offset} \\[0.3em]
     \sum_{k=1}^{2m} \lambda_{i,j,k} &= 1,
    &&\forall i,j \in \mathcal{N},\ i<j
    \label{eq:polygon-farkas-sum} \\
     \sum_{k=0}^{m-1} \lambda_{i,j,k+1} a_{i,k}
      + \sum_{k=0}^{m-1} \lambda_{i,j,k+m+1} a_{j,k} &= 0,
    &&\forall i,j \in \mathcal{N},\ i<j
    \label{eq:polygon-farkas-x} \\
     \sum_{k=0}^{m-1} \lambda_{i,j,k+1} b_{i,k}
      + \sum_{k=0}^{m-1} \lambda_{i,j,k+m+1} b_{j,k} &= 0,
    &&\forall i,j \in \mathcal{N},\ i<j
    \label{eq:polygon-farkas-y} \\
     \sum_{k=0}^{m-1} \lambda_{i,j,k+1} s_{i,k}
      + \sum_{k=0}^{m-1} \lambda_{i,j,k+m+1} s_{j,k} &\ge 0,
    &&\forall i,j \in \mathcal{N},\ i<j
    \label{eq:polygon-farkas-sep} \\[0.3em]
     0 \le \theta_i &\le \phi_m,
    &&\forall i \in \mathcal{N}
    \label{eq:polygon-rotation} \\
     \lambda_{i,j,k} &\ge 0,
    &&\forall i,j \in \mathcal{N},\ i<j, \nonumber \\
    &&&k \in \{1,\ldots,2m\} \\[0.3em]
     (x_i - x_j)^2 + (y_i - y_j)^2 &\ge (2\rho_m)^2,
    &&\forall i,j \in \mathcal{N},\ i<j
    \label{eq:polygon-dist} \\
     R &\ge R_{\min}
    \label{eq:polygon-bounds}
    \end{alignat}
\end{subequations}
Constraints \eqref{eq:polygon-containment} ensure that all vertices of each inner
$m$-gon lie within the outer $\ell$-gon. The vertex angles $\delta_{m,j}$ account for
the positions of the inner polygon vertices relative to their inward edge normals.
Each vertex must satisfy all $\ell$ half-space constraints defining the outer $\ell$-gon.

Constraints \eqref{eq:polygon-normal-x}--\eqref{eq:polygon-offset} define the
oriented half-space representation of the inner polygons. For this, each edge $j \in \mathcal{M}$ of
inner polygon $i \in \mathcal{N}$ is evaluated by the oriented inward normal $(a_{i,j}, b_{i,j})$ and
offset $s_{i,j}$, representing the defining half-space constraint $a_{i,j} x + b_{i,j} y \geq s_{i,j}$.

Constraints \eqref{eq:polygon-farkas-sum}--\eqref{eq:polygon-farkas-sep} represent the Farkas-based separation
conditions as described above.

Constraint \eqref{eq:polygon-dist} requires that the centers of any two inner $m$-gons are at least $2\rho_m$ apart, which is redundant given the Farkas separation conditions.

Constraint \eqref{eq:polygon-rotation} restricts the rotation angle to $[0, \phi_m]$
due to the $m$-fold rotational symmetry of regular $m$-gons, i.e., rotations beyond
$\phi_m$ are congruent to rotations within this range.
Finally, the lower bound $R_{\min}$ in \eqref{eq:polygon-bounds} is derived from the containment requirement that the
area of the scaled outer $\ell$-gon must be at least the total area of the $n$ unit inner $m$-gons.

\subsection{Computational Results} 
\label{subsec:polygonresults}

The resulting circumradii for all computed polygon-packing instances are listed in the online supplement\footnote{\url{https://github.com/DominikKamp/Packing}}. We report the improving solutions in Table~\ref{tab:polygon} and depict them in Figure~\ref{fig:polygon}.

\begin{table}[ht]
\caption{Improving solutions for polygon packing. References are provided if an accompanying publication is available. Otherwise the name and the year are taken from Erich Friedman's packing website.}\label{tab:polygon}
\begin{tabular}{@{}rrrllr@{}}
\toprule
$\ell$ & $m$ & $n$ & Circumradius & Previous & Source\\
\midrule
4 & 3 & 12 & 3.13\textbf{403} & 3.13802 & Cantrell 2002 \\
5 & 3 & 6 & 2.05\textbf{332} & 2.05521 & Morandi 2012 \\
5 & 3 & 10 & 2.\textbf{67560} & 2.71218 & Cantrell 2012 \\
5 & 3 & 11 & 2.\textbf{75788} & 2.80900 & Cantrell 2012 \\
5 & 3 & 12 & 2.9\textbf{0916} & 2.93400 & Cantrell 2012 \\
5 & 3 & 13 & \textbf{2.97298} & 3.00600 & Cantrell 2012 \\
5 & 3 & 14 & 3.\textbf{07677} & 3.13711 & Morandi 2012 \\
5 & 4 & 9 & 3.2\textbf{4086} & 3.28500 & Hirsh 2022 \\
5 & 4 & 10 & 3.3\textbf{6024} & 3.38000 & Cantrell 2012 \\
5 & 4 & 11 & 3.\textbf{47760} & 3.51400 & Cantrell 2012 \\
5 & 4 & 12 & 3.57\textbf{255} & 3.57700 & Cantrell 2012 \\
5 & 4 & 13 & 3.75\textbf{495} & 3.75575 & Cantrell 2012 \\
6 & 3 & 10 & 2.59\textbf{297} & 2.59808 & Friedman 2005 \\
6 & 3 & 12 & 2.81\textbf{217} & 2.81900 & Morandi 2008 \\
6 & 3 & 13 & 2.88\textbf{480} & 2.88676 & Friedman 2005 \\
6 & 6 & 11 & 3.9\textbf{2451} & 3.93010 & \cite{georgiev2025mathematical,novikov2025alphaevolve} \\
6 & 6 & 12 & 3.941\textbf{65} & 3.94192 & \cite{georgiev2025mathematical,novikov2025alphaevolve} \\
6 & 6 & 14 & 4.2\textbf{6900} & 4.27240 & \cite{friedman2015hexagon} \\
6 & 6 & 15 & 4.4\textbf{4728} & 4.45406 & \cite{friedman2015hexagon} \\
6 & 6 & 16 & 4.5\textbf{2788} & 4.53633 & \cite{friedman2015hexagon} \\
6 & 6 & 17 & 4.61\textbf{362} & 4.61881 & \cite{friedman2015hexagon} \\
6 & 6 & 23 & 5.4\textbf{0001} & 5.42858 & \cite{friedman2015hexagon} \\
\bottomrule
\end{tabular}
\end{table}

\noindent\begin{minipage}{\linewidth}
\begin{center}
\includegraphics[width=0.20\linewidth]{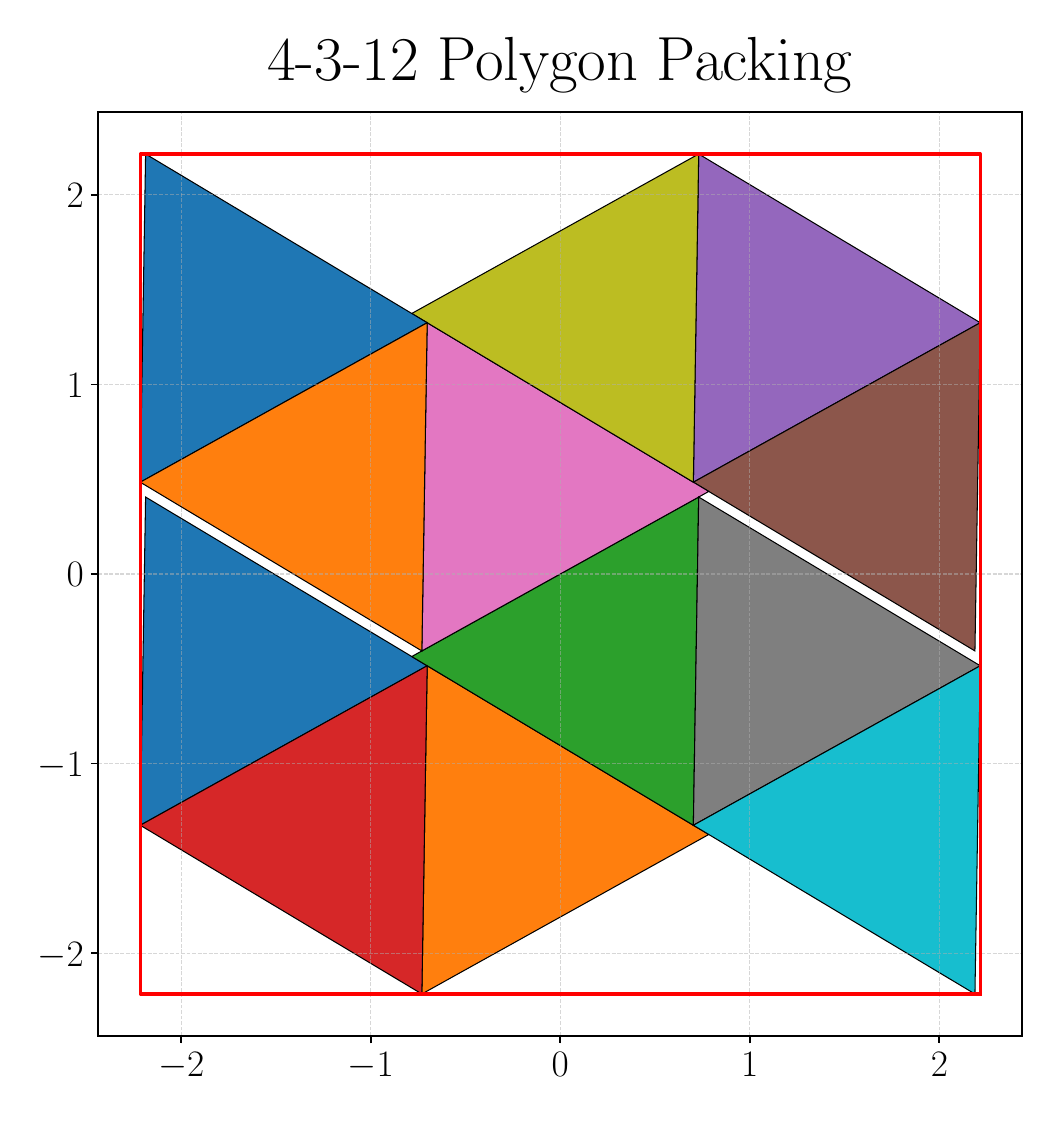}
\includegraphics[width=0.20\linewidth]{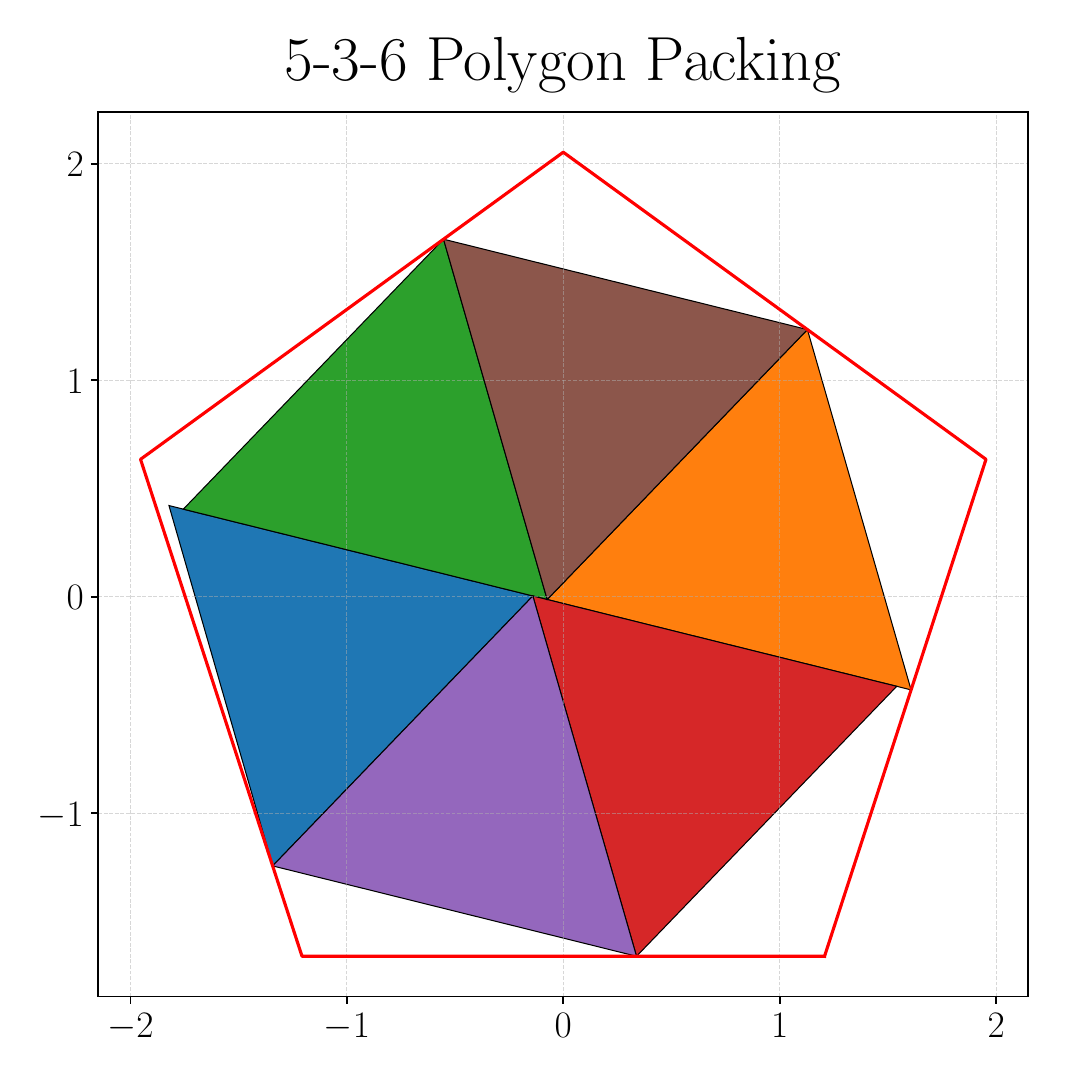}
\includegraphics[width=0.20\linewidth]{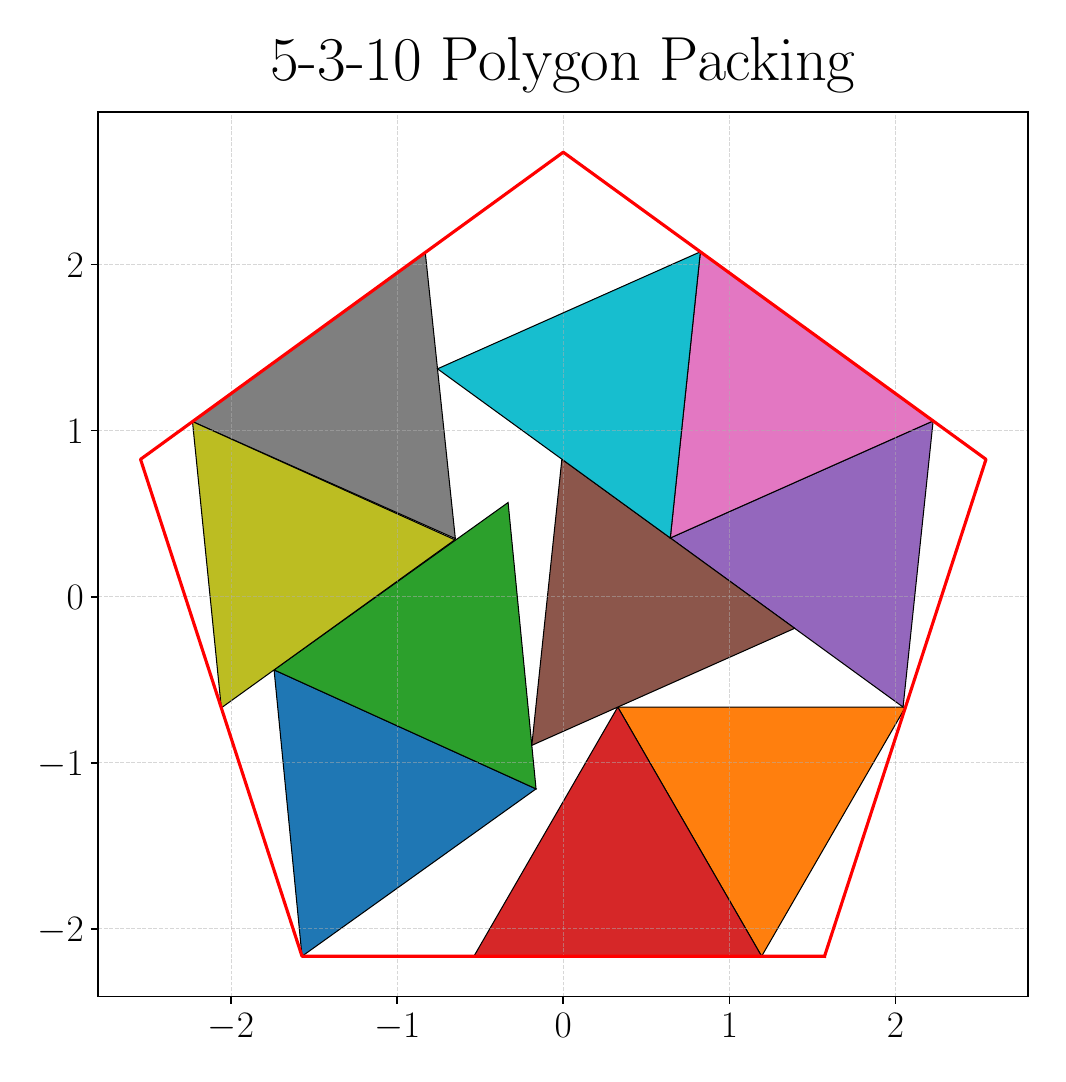}
\includegraphics[width=0.20\linewidth]{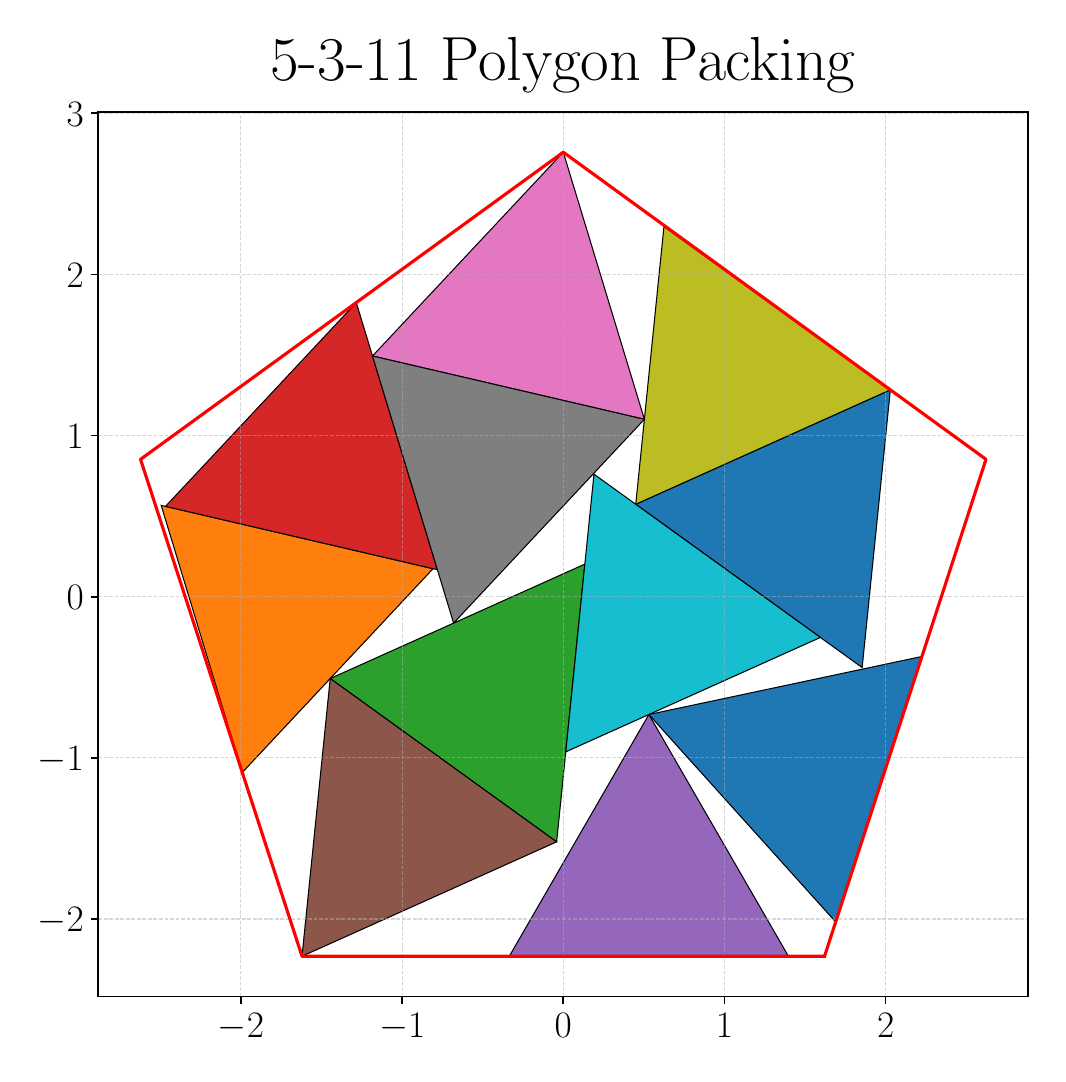}\\
\includegraphics[width=0.20\linewidth]{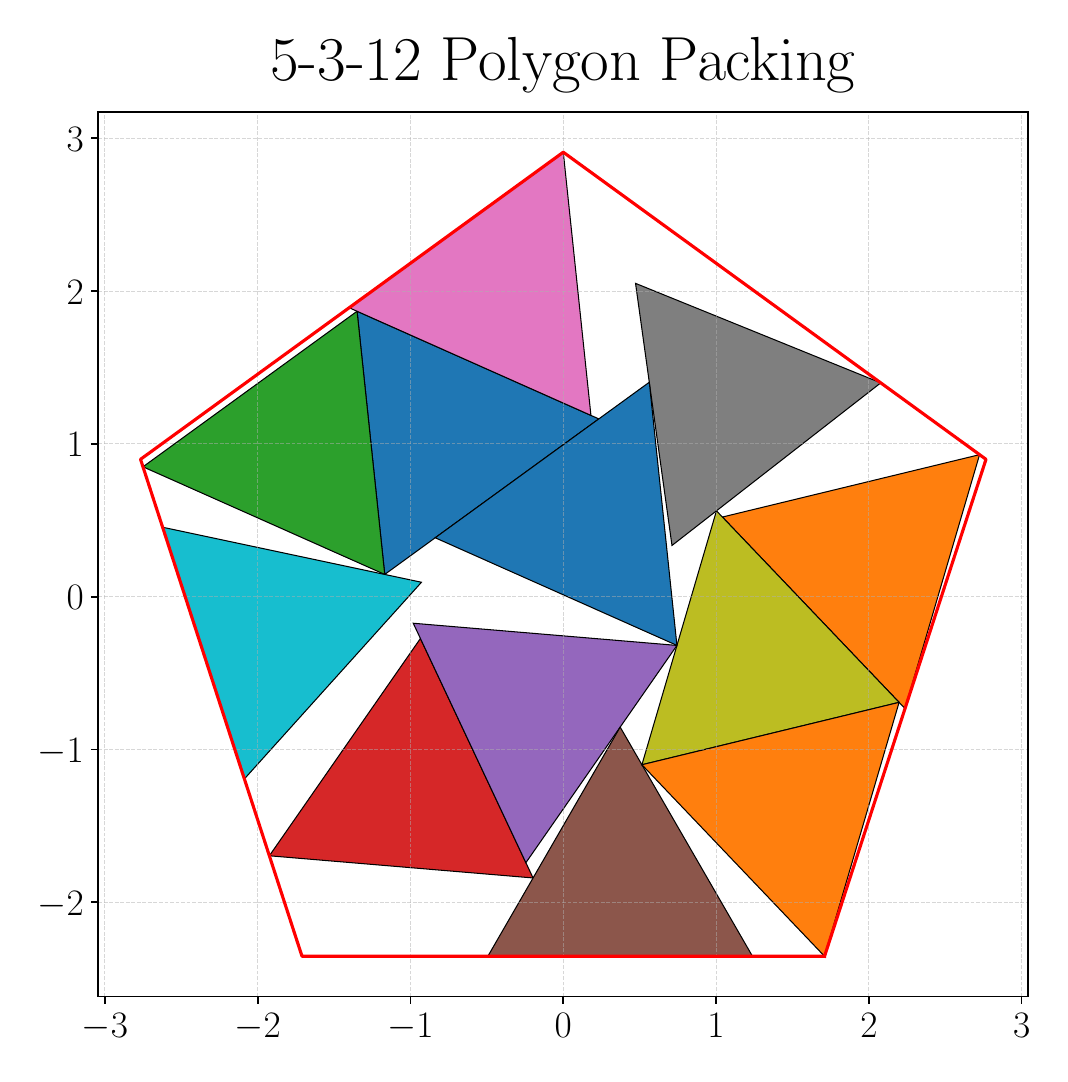}
\includegraphics[width=0.20\linewidth]{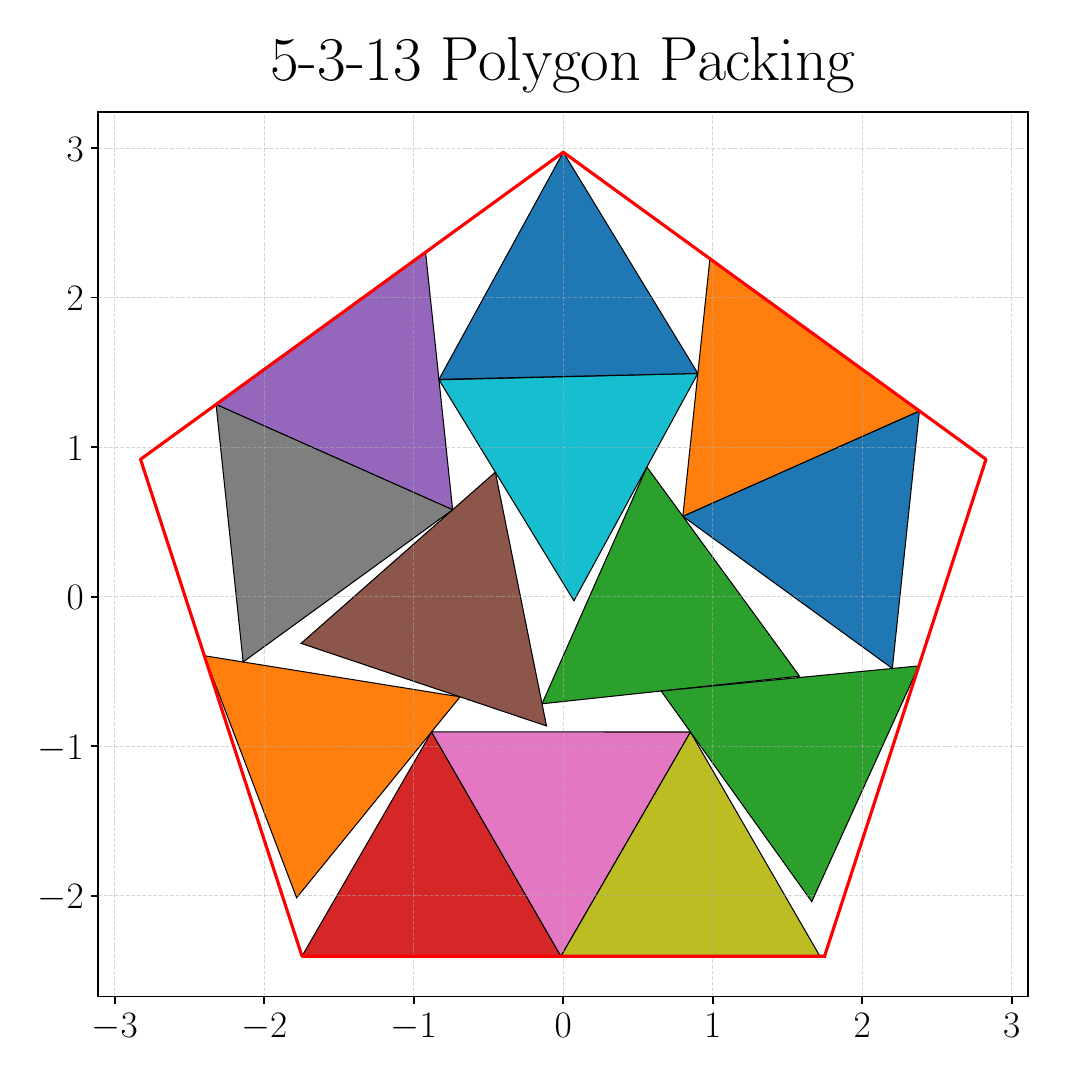}
\includegraphics[width=0.20\linewidth]{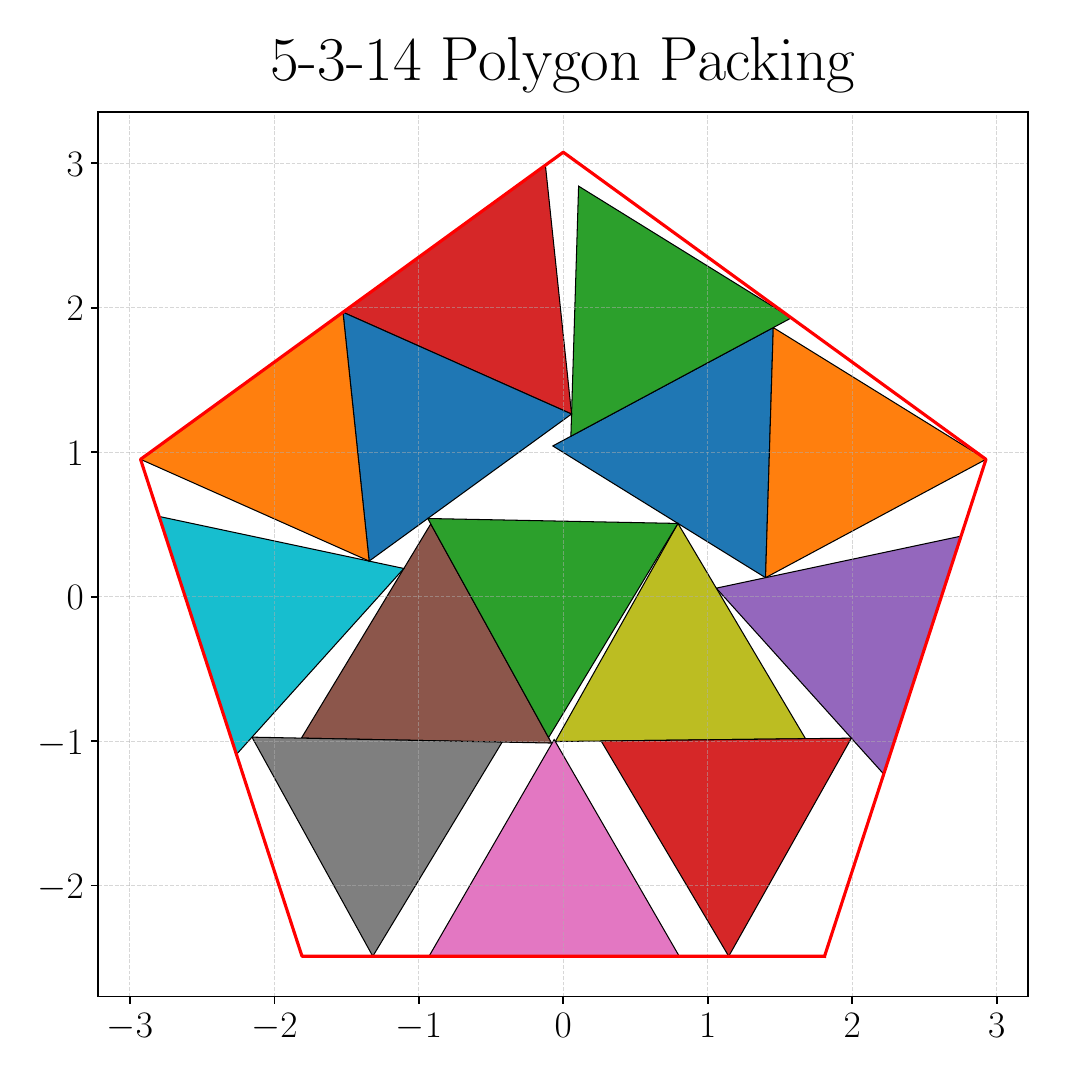}
\includegraphics[width=0.20\linewidth]{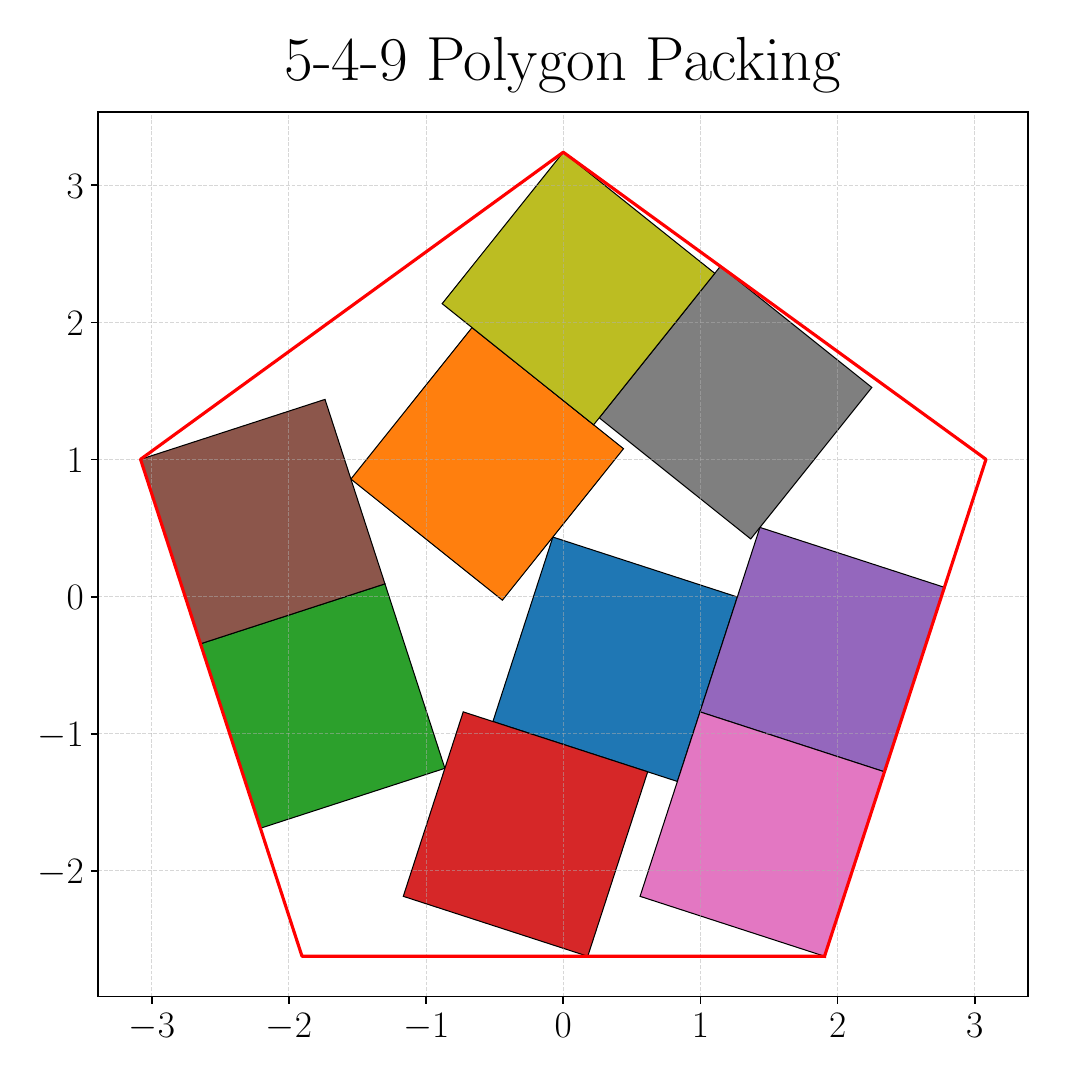}\\
\includegraphics[width=0.20\linewidth]{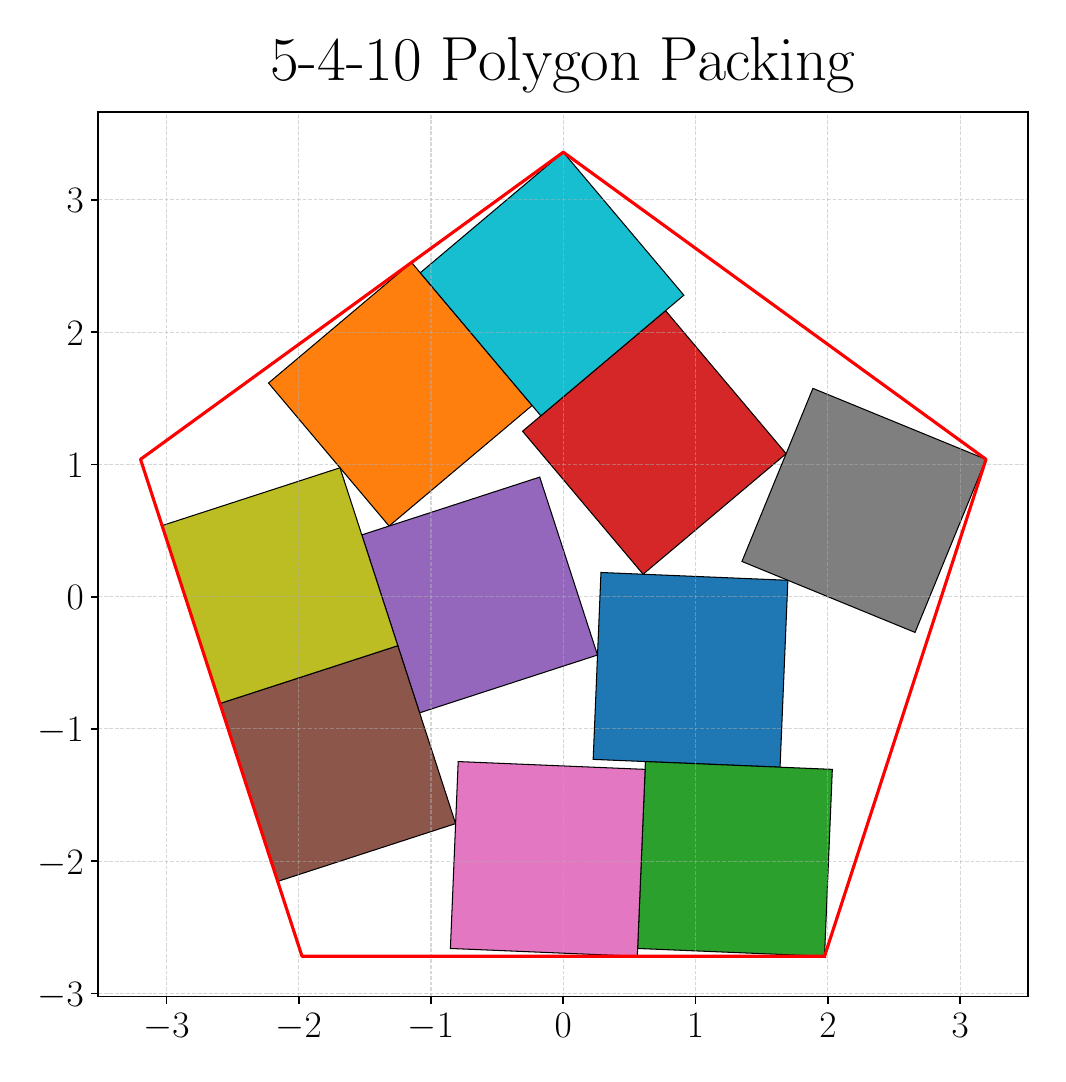}
\includegraphics[width=0.20\linewidth]{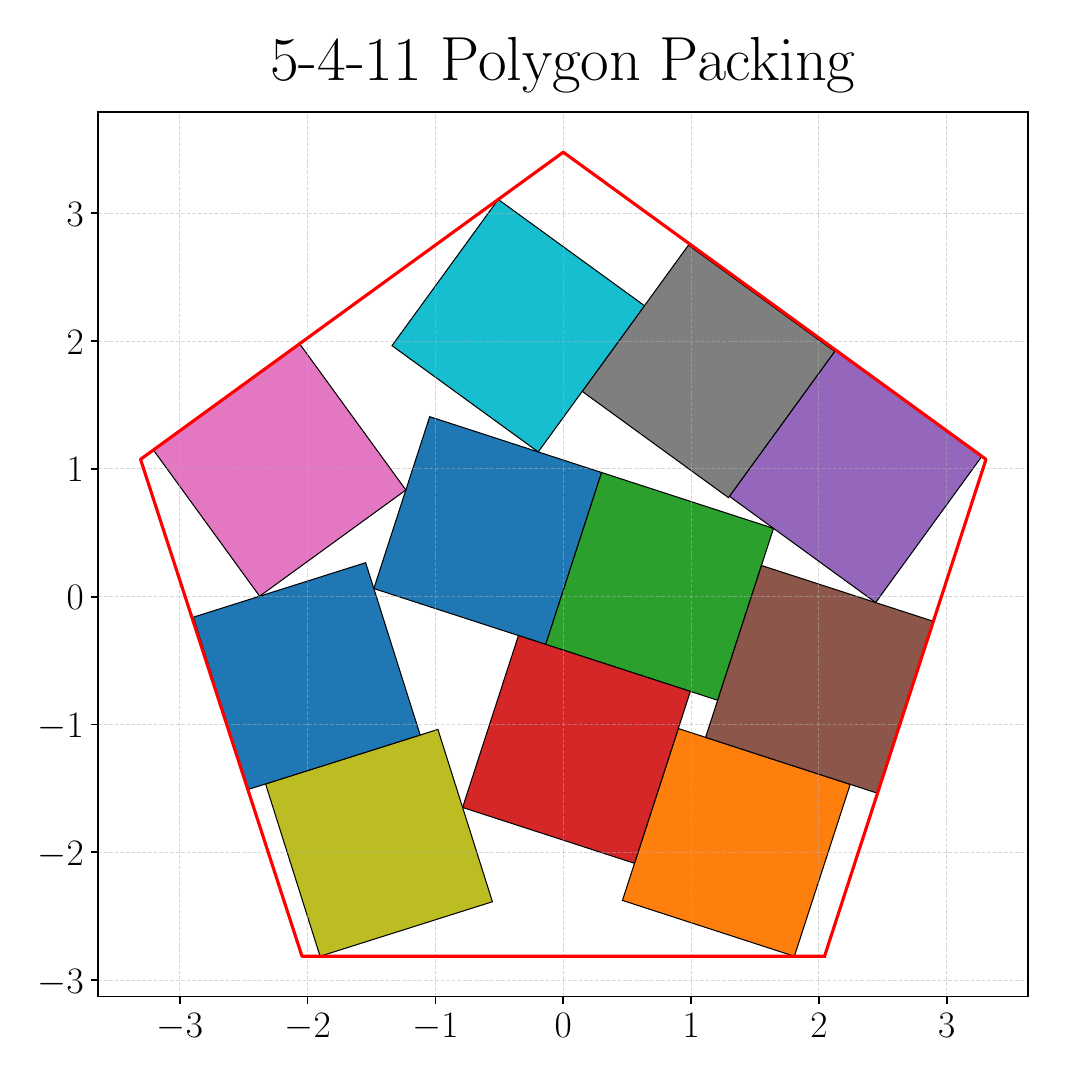}
\includegraphics[width=0.20\linewidth]{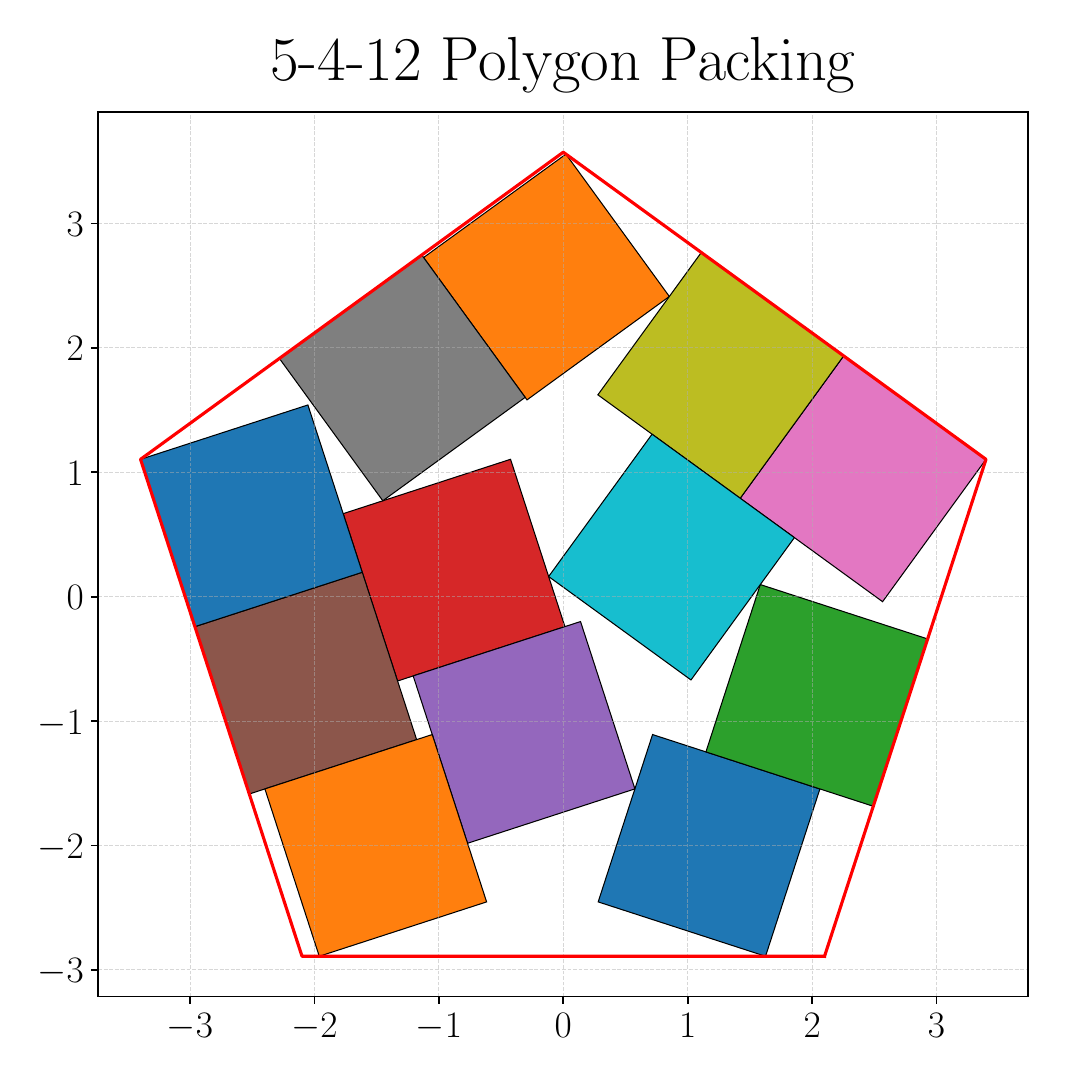}
\includegraphics[width=0.20\linewidth]{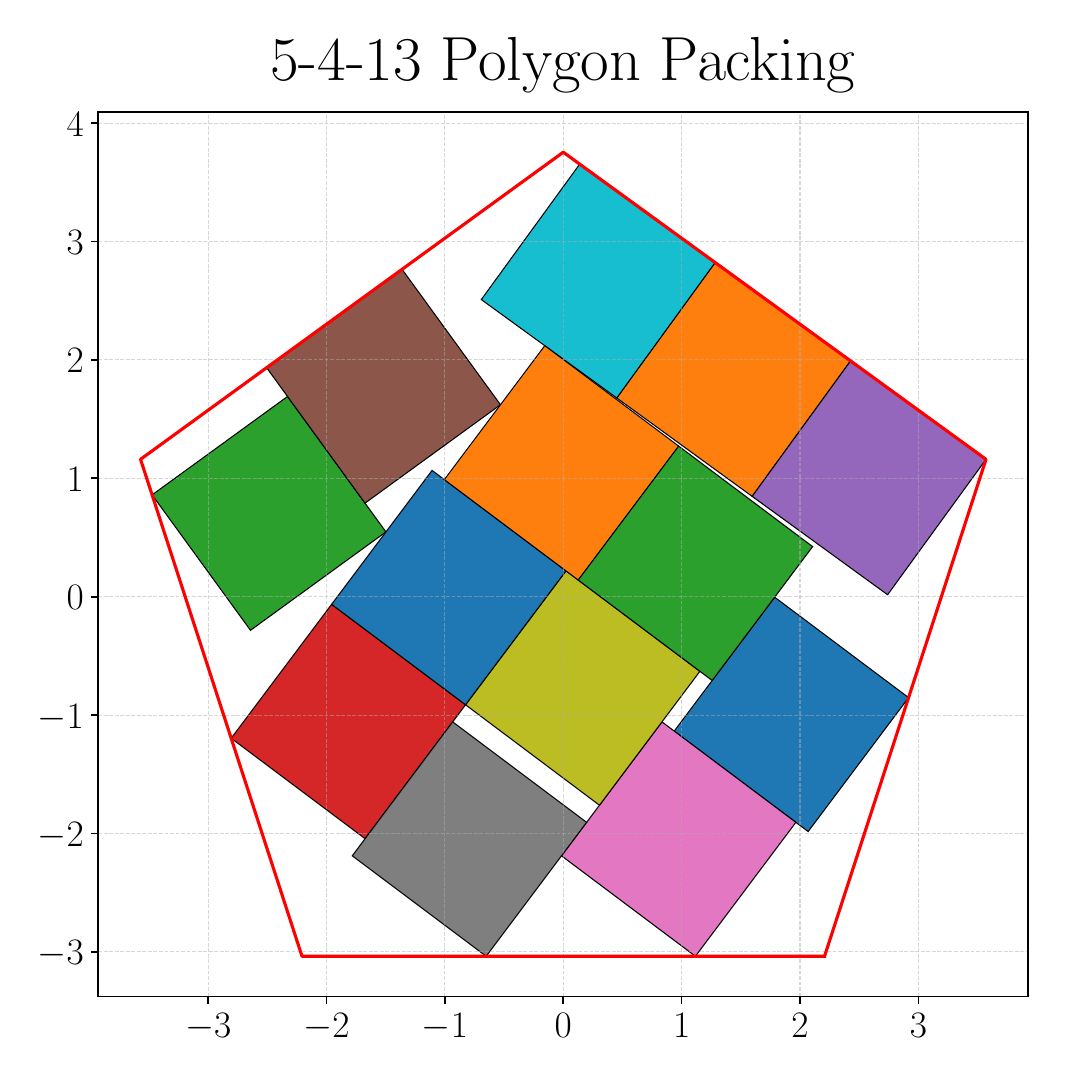}\\
\includegraphics[width=0.20\linewidth]{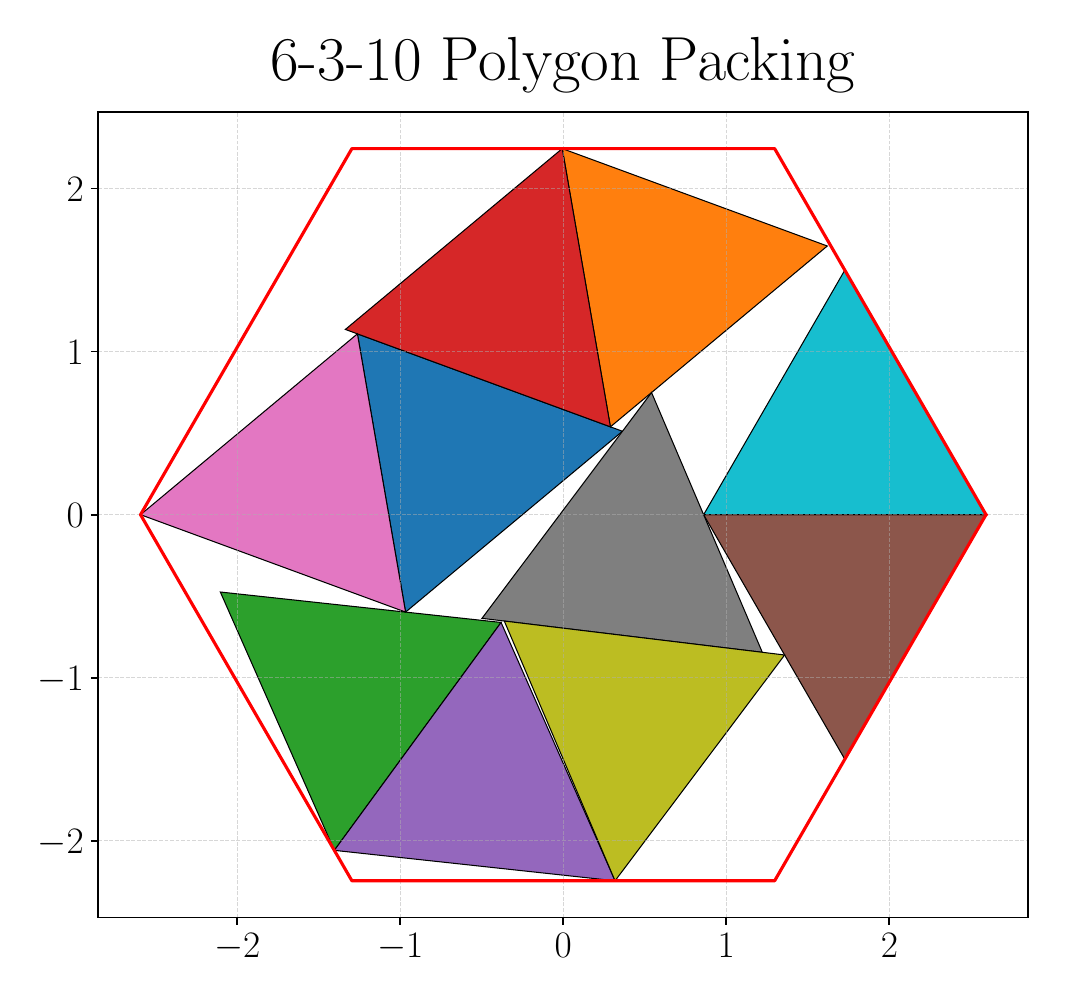}
\includegraphics[width=0.20\linewidth]{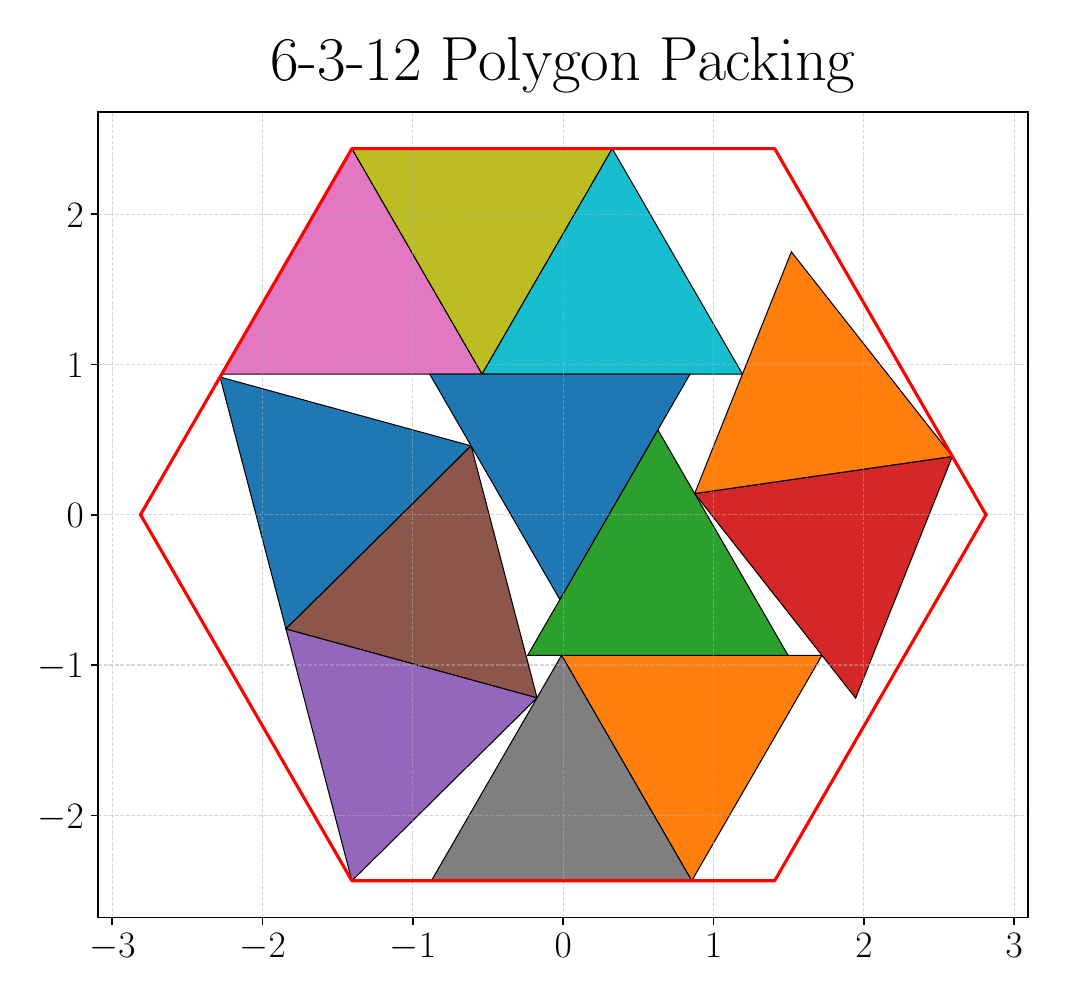}
\includegraphics[width=0.20\linewidth]{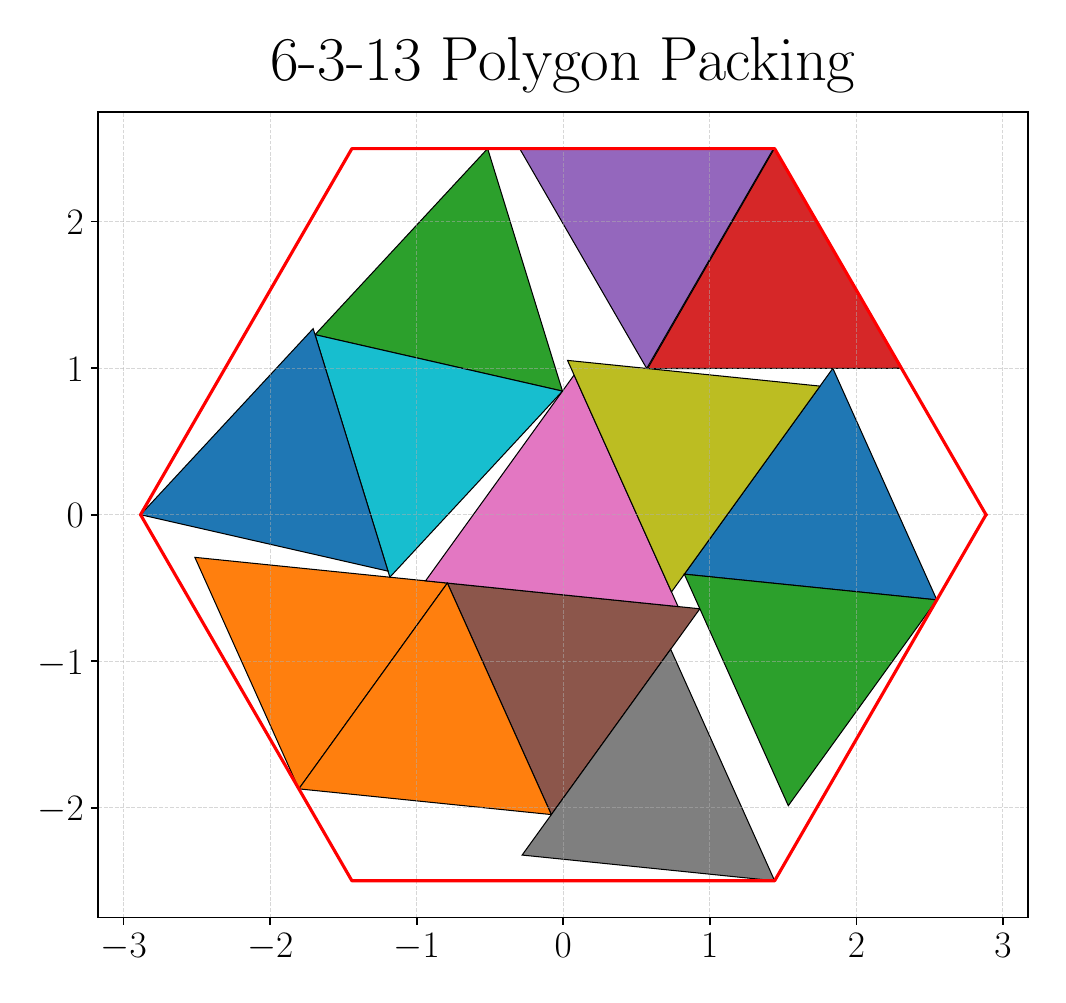}
\includegraphics[width=0.20\linewidth]{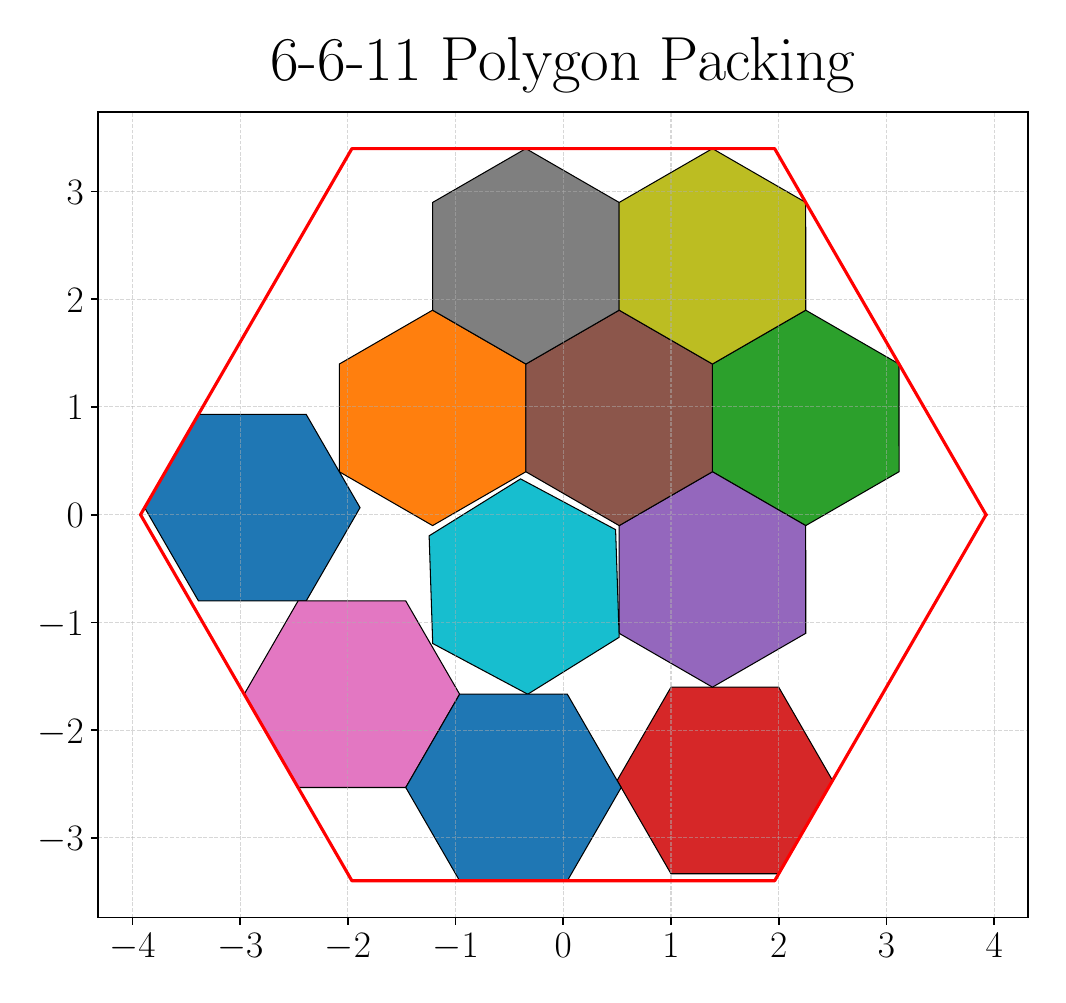}\\
\includegraphics[width=0.20\linewidth]{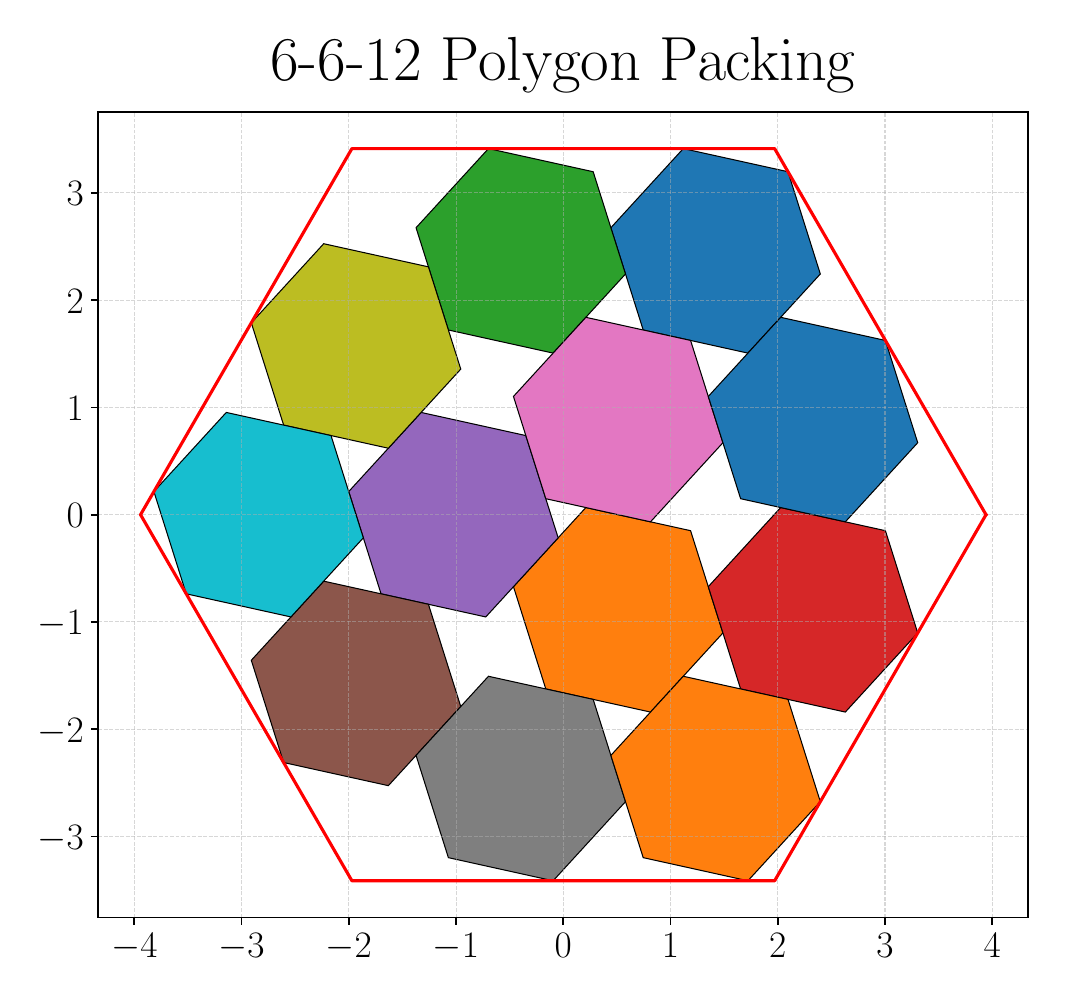}
\includegraphics[width=0.20\linewidth]{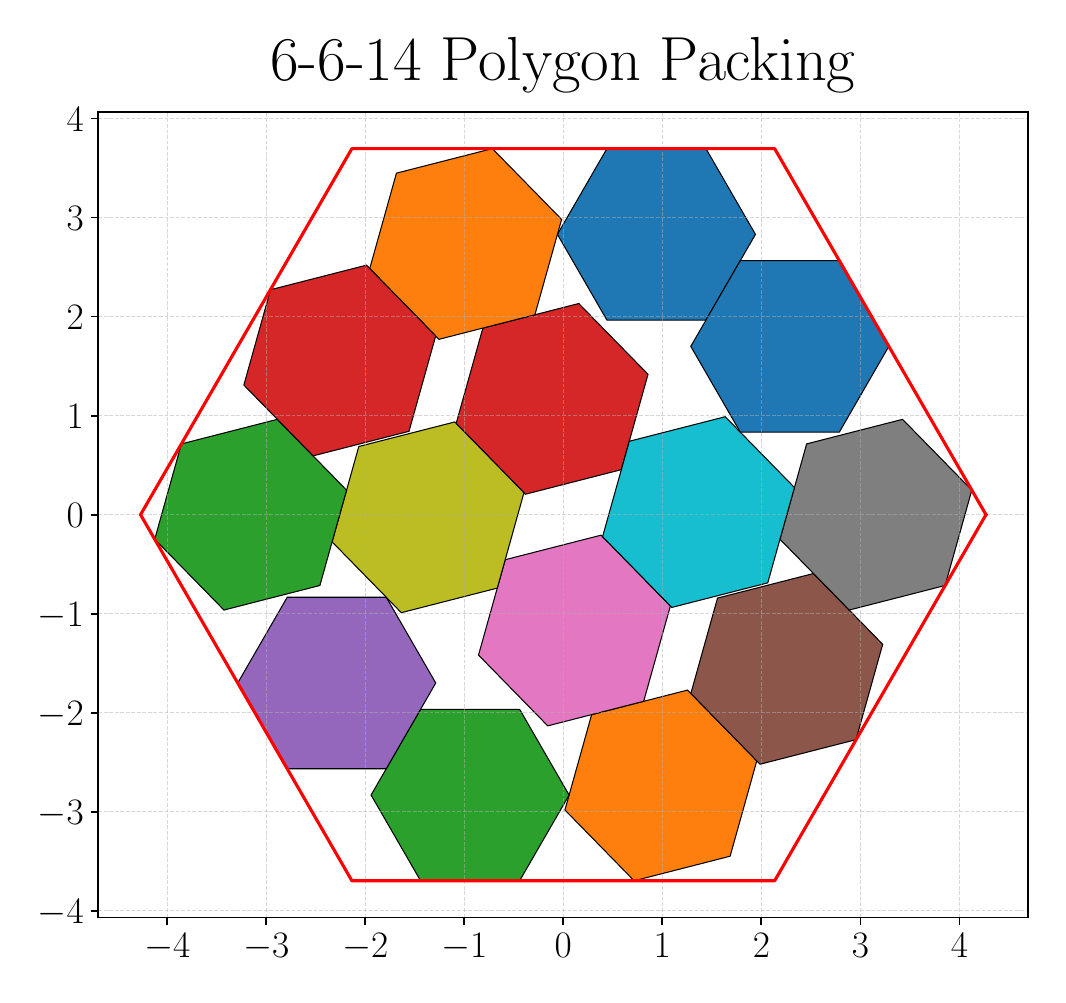}
\includegraphics[width=0.20\linewidth]{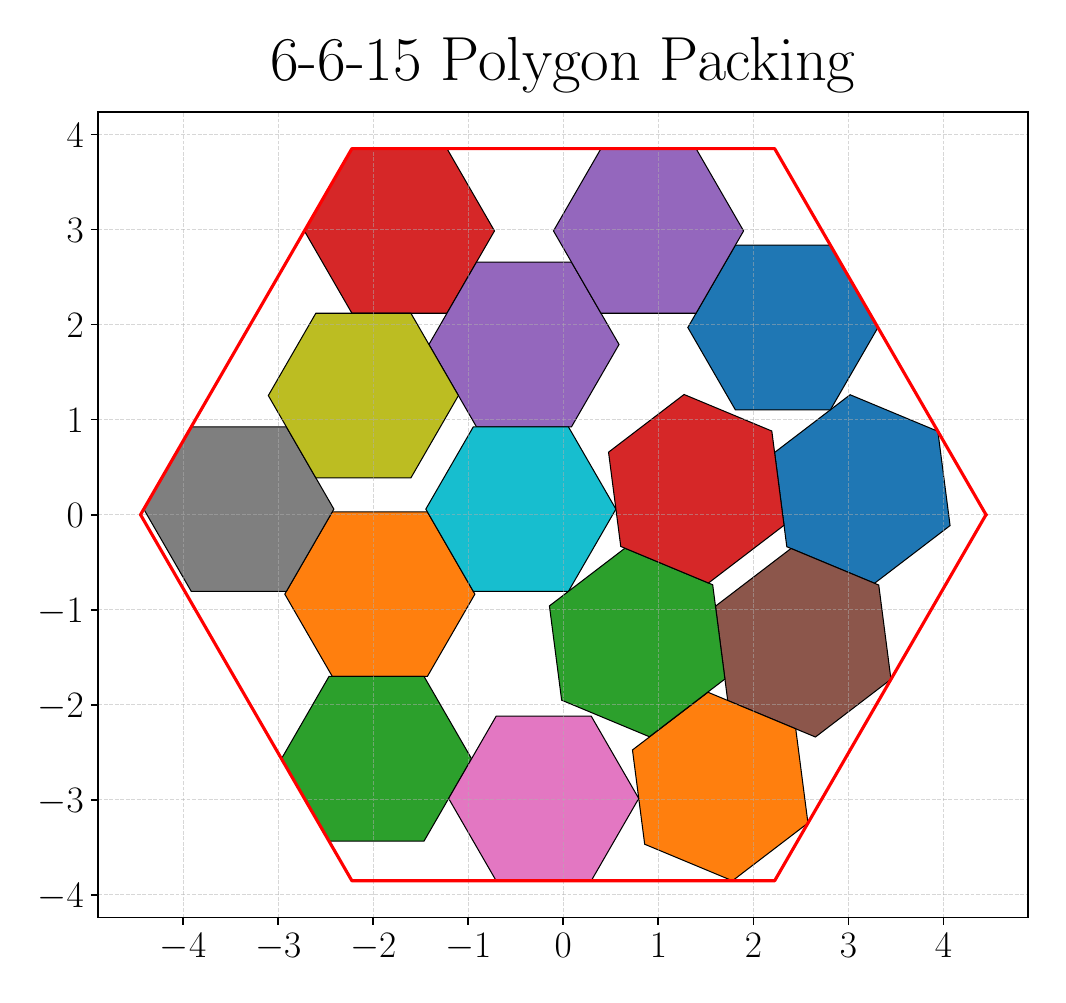}
\includegraphics[width=0.20\linewidth]{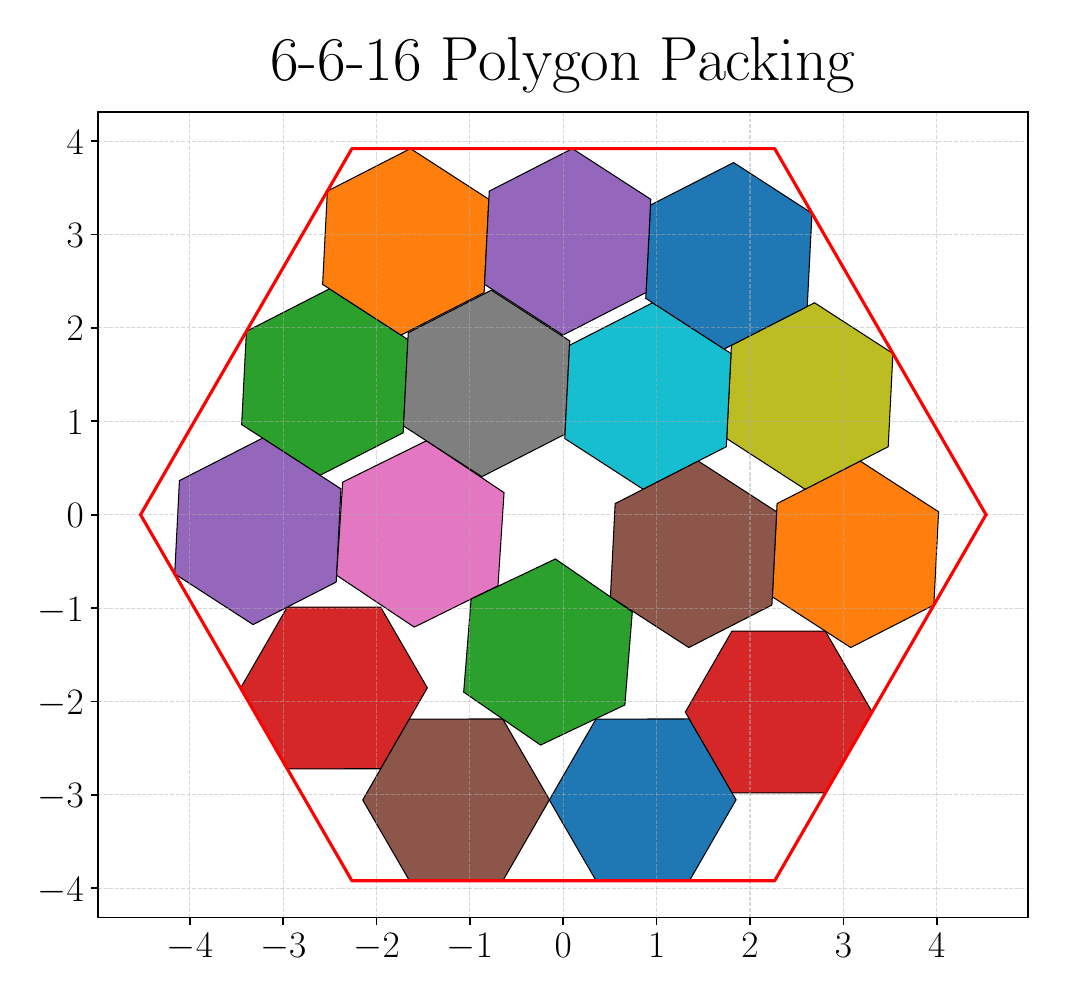}\\
\includegraphics[width=0.20\linewidth]{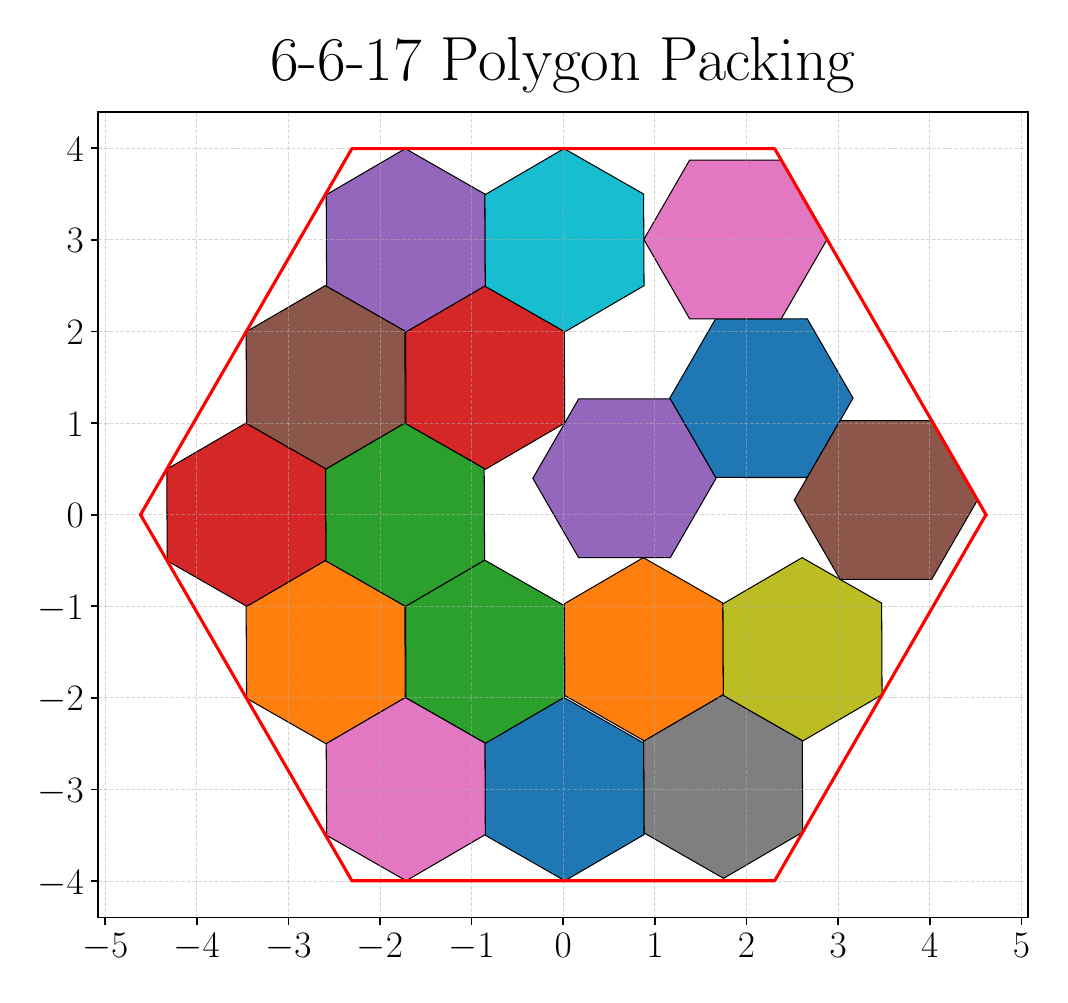}
\includegraphics[width=0.20\linewidth]{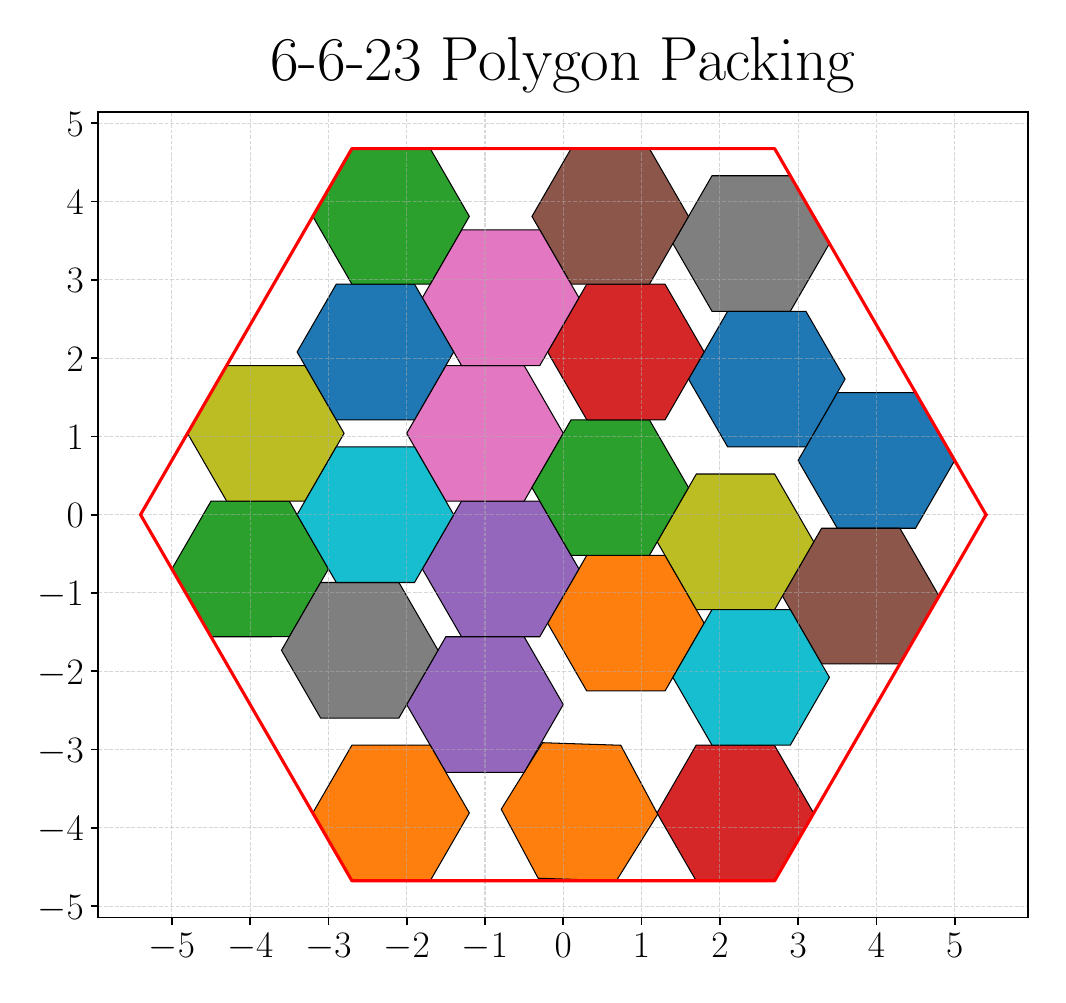}
\end{center}
\captionof{figure}{Graphical representation of the solutions for the polygon-packing problem}\label{fig:polygon}
\end{minipage}

The median optimality gap across $(\ell,m)$ groups ranges from $9.9\%$ to $15.2\%$ (average $13.0\%$), taking the best primal and dual bounds across all solver setups.
For small $n$, the setups tighten the dual bound beyond the area lower bound $R_{\min}$, proving optimality for one and two inner elements and several instances with $n \leq 5$, while for $n \geq 6$, the dual bound generally remains at $R_{\min}$.

\subsection{Solution Polishing}\label{sec:polypolish}
 
The circumradius of each inner polygon is scaled outward by a factor of $1 + \varepsilon / \rho_m$ in the containment constraints, and the inradius is increased by $\varepsilon / 2$ in the separation constraints. The two different scalings ensure that the absolute geometric margin is the same~($\varepsilon$) between an inner element and the outer boundary as between two inner elements, so that any sufficiently precise solution with respect to solver tolerances is feasible to machine tolerance. For our runs $\varepsilon=10^{-8}$.

After the solutions have been obtained, the rotation angles remain untouched and all center coordinates are scaled uniformly with fixed origin by a carefully chosen factor $\sigma > 0$.
The scaling factor $\sigma$ is computed as the minimum value that eliminates all pairwise overlaps, determined via the separating axis theorem using edge normals. Additionally, downscaling towards the origin is restricted by the containment condition for the given~$R$ to avoid unnecessary objective increase.
The circumradius~$R$ is then recomputed from the scaled configuration using projections of the inner elements on the face normals of the outer element.
A relative safety margin of $10^{-14}$ is applied to account for numerical rounding errors.

\subsection{Model Variants}\label{subsubsec:polygonvariants}

To investigate the impact of formulation choices on solver performance, we test five variants of the polygon-packing model.
The formulation \eqref{eq:polygon-obj}--\eqref{eq:polygon-bounds} described above, including the redundant distance constraint~\eqref{eq:polygon-dist}, serves as the base and is denoted \textsc{Dist}.
The remaining four variants modify or extend it as follows.
\begin{description}\itemsep5pt
\item[No distance constraint \textsc{(Nodist)}:] The formulation \textsc{Nodist} removes the redundant distance constraint~\eqref{eq:polygon-dist} from the base formulation.

\item[Inner separation \textsc{(Inner)}:] This variant replaces the Farkas-based separation \eqref{eq:polygon-farkas-sum}--\eqref{eq:polygon-farkas-sep} entirely with direct separating hyperplane constraints.
For each pair $(i,j)$ with $i < j$, a separating direction angle $\alpha_{i,j} \in [0, 2\pi]$ and offset $d_{i,j} \in \mathbb{R}$ are introduced as decision variables in place of the Farkas multipliers.
All vertices of polygon~$i$ must lie on the upper side of the hyperplane and all vertices of polygon~$j$ on the lower side:
\begin{subequations}
\begin{align}
  \sin(\alpha_{i,j})\bigl(x_i + \sin(\theta_i + \delta_{m,k})\bigr)
  + \cos(\alpha_{i,j})\bigl(y_i + \cos(\theta_i + \delta_{m,k})\bigr) &\geq d_{i,j}
  \label{eq:inner-sep-upper} \\
  \sin(\alpha_{i,j})\bigl(x_j + \sin(\theta_j + \delta_{m,k})\bigr)
  + \cos(\alpha_{i,j})\bigl(y_j + \cos(\theta_j + \delta_{m,k})\bigr) &\leq d_{i,j}
  \label{eq:inner-sep-lower}
\end{align}
\end{subequations}
for all $i,j \in \mathcal{N}$ with $i < j$ and $k \in \mathcal{M}$.
This formulation uses fewer variables per pair ($2$ instead of $2m$ Farkas multipliers) but more constraints ($2m$ instead of $4$ Farkas constraints including the normalization constraint).

\item[Farkas normalization \textsc{(Farkas)}:] This variant replaces the overall normalization of Farkas multipliers~\eqref{eq:polygon-farkas-sum} with a lower bound for each polygon:
\begin{subequations}
\begin{align}
  \sum_{k \in \mathcal{M}} \lambda_{i,j,k+1} &\geq 1
  \label{eq:polygon-farkas-perpolygon-i} \\
  \sum_{k \in \mathcal{M}} \lambda_{i,j,k+m+1} &\geq 1
  \label{eq:polygon-farkas-perpolygon-j}
\end{align}
\end{subequations}
for all $i,j \in \mathcal{N}$ with $i < j$.
The Farkas sum for each polygon needs to be positive because only the intersection of edge constraints across both full-dimensional polygons can be lower-dimensional to prove separation.
The remaining Farkas constraints \eqref{eq:polygon-farkas-x}--\eqref{eq:polygon-farkas-sep} are unchanged.

\item[Symmetry breaking \textsc{(Sym)}:] This variant extends \textsc{Dist} with elementary symmetry-breaking constraints.
By permutation of variable assignments, the coordinates can always be ordered in nondecreasing order of the $x$ coordinate as in~\eqref{eq:symsortx}.
Furthermore, the $\ell$-fold rotational and reflection symmetries of the outer polygon and of each inner polygon assignment along the outer polygon axes make it possible to restrict the centroid of all polygon centers to the cone spanned by the first half of the first angular sector:
\begin{subequations}
\begin{align}
  \sum_{i \in \mathcal{N}} x_i &\geq 0
  \label{eq:polygon-sym-cx} \\
  -\cos(\phi_\ell/2) \sum_{i \in \mathcal{N}} x_i + \sin(\phi_\ell/2) \sum_{i \in \mathcal{N}} y_i &\geq 0
  \label{eq:polygon-sym-cy}
\end{align}
\end{subequations}
Since permutation of elements does not change the coordinate sums, both symmetry reductions can be applied simultaneously.
\end{description}

Table~\ref{tab:polygon_overall} reports, across all instances and for each formulation, solutions found (\emph{found}), median and mean per-instance relative difference to \textsc{Dist}, and counts of instances where a formulation strictly beats (\emph{better}) or is strictly beaten by (\emph{worse}) \textsc{Dist}.

\begin{longtable}{rrrrrr}
  \caption{Polygon packing: overall formulation comparison across 750 instances. Solutions found, median and mean relative difference to \textsc{Dist}, and diff counts per formulation.} \label{tab:polygon_overall} \\
  \toprule
  metric & \textsc{Dist} & \textsc{Nodist} & \textsc{Inner} & \textsc{Farkas} & \textsc{Sym} \\
  \midrule
  \endfirsthead
  \multicolumn{6}{l}{\tablename\ \thetable{} -- \textit{continued}} \\
  \toprule
  metric & \textsc{Dist} & \textsc{Nodist} & \textsc{Inner} & \textsc{Farkas} & \textsc{Sym} \\
  \midrule
  \endhead
  \endfoot
  \bottomrule
  \endlastfoot
  found & 734 & 416 & 358 & 329 & 737 \\
  median & +0.000\% & +0.000\% & +0.575\% & +0.000\% & +2.178\% \\
  mean & +0.000\% & +0.897\% & +13.599\% & +0.790\% & +3.698\% \\
  better & 0 & 4 & 1 & 1 & 2 \\
  worse & 0 & 187 & 200 & 139 & 539 \\
\end{longtable}

The redundant distance constraints have the strongest impact, so we use \textsc{Dist} as the reference formulation.
The formulation \textsc{Nodist}, without distance constraints, is strictly worse on~187 instances versus only~4 where it is better.
The alternative separation approaches \textsc{Inner} and \textsc{Farkas} produce weaker results, while \textsc{Inner} additionally faces a~$0.575\%$ median ($13.6\%$ mean) quality degradation.
\textsc{Sym} produces the worst solution quality overall: the median objective is $2.178\%$ ($3.7\%$ mean) worse than~\textsc{Dist}, and~539 instances are strictly worse versus only~2 where it is better.
This suggests that symmetry breaking impedes performance of heuristics by unnecessarily enforcing certain orders and locations to accept a configuration.

\section{Packing Platonic Solids} 
\label{sec:platonicpacking}

The problem of packing $n$ identical Platonic solids of a certain type into a Platonic solid of the same or different type of minimum volume\footnote{Or, equivalently, surface area, or circumradius.}
extends the polygon-packing problem to three dimensions, where each inner solid can rotate freely and non-overlap
conditions must account for the specific polyhedral geometry of each element.
The Farkas-based formulation of Section~\ref{sec:polygonpacking} extends naturally to three dimensions, replacing edge
normals by face normals and a single rotation angle by a three-dimensional rotational parametrization.

\subsection{Optimization Model} 
\label{subsec:platonicmodel}

The problem is to pack $n$ Platonic solids of type $m$ with unit circumradius into a Platonic solid of type $\ell$,
minimizing the circumradius $R$ of the outer solid, where the types $1, \ldots, 5$ correspond to the five Platonic solids: tetrahedron, octahedron, cube, icosahedron, and dodecahedron, respectively.
Let $\mathcal{N} = \{1, 2, \ldots, n\}$ denote the set of inner solids,
$\mathcal{V}_m$ the set of vertices for inner solid type $m$,
$\mathcal{F}_m$ the set of faces for inner solid type $m$, and
$\mathcal{F}_\ell$ the set of faces for outer solid type $\ell$.
We use the following geometric constants for Platonic solids with unit circumradius:
\begin{description}
\item[$\mathbf{v}^j_m \in \mathbb{R}^3$:] standard position for vertex $j \in \mathcal{V}_m$
\item[$\mathbf{n}^f_m \in \mathbb{R}^3$:] standard inward normal for face $f \in \mathcal{F}_m$
\item[$\mathbf{n}^k_\ell \in \mathbb{R}^3$:] standard inward normal for face $k \in \mathcal{F}_\ell$
\item[$\rho_m$:] inradius of inner solid type $m$
\item[$\rho_\ell$:] inradius of outer solid type $\ell$
\item[$R_{\min} = (n \cdot V_m / V_\ell)^{1/3}$:] lower bound on circumradius from volume requirement, where $V_m$, $V_\ell$ are the unit volumes of solid type $m$ and $\ell$
\end{description}
We define the following decision variables:
\begin{description}
\item[$R \ge 0$:] circumradius of the scaled outer solid (to be minimized)
\item[$(x_i, y_i, z_i) \in \mathbb{R}^3$:] coordinates of the center of inner solid $i \in \mathcal{N}$
\item[$\theta_i, \iota_i, \kappa_i \in \lbrack 0, 2\pi \rbrack$:] ZYX rotation angles of inner solid $i \in \mathcal{N}$
\item[$a_{i,f}, b_{i,f}, c_{i,f} \in \mathbb{R}$:] oriented inward normal components for face $f \in \mathcal{F}_m$ of inner solid $i \in \mathcal{N}$
\item[$e_{i,f} \in \mathbb{R}$:] offset for face $f \in \mathcal{F}_m$ of inner solid $i \in \mathcal{N}$
\item[$\lambda_{i,j,k} \ge 0$:] Farkas multiplier for pair $i, j \in \mathcal{N}$ with $i < j$ and common face index $k \in \{1, \ldots, 2|\mathcal{F}_m|\}$
\end{description}
The pairwise non-overlap conditions are formulated using the same Farkas certificate as for the polygon case,
extended to three dimensions.
The certificate carries over analogously: two convex polyhedra described by their face half-space representations have
disjoint interiors if and only if there exist nonnegative multipliers for their combined face inequalities, summing to
one, that yield a zero normal sum and a nonnegative offset sum.
For each pair $(i,j)$ with $i < j$, the $2|\mathcal{F}_m|$ Farkas multipliers $\lambda_{i,j,k}$, one per face of each
of the two inner solids, encode this certificate; the geometric interpretation and LP-based reverse mapping of
Section~\ref{subsec:polygonmodel} carry over to three dimensions analogously.
The approach applies to any convex polyhedron without shape-specific derivations.
Since the LP has three equality constraints (one per component of the normal equation in $\mathbb{R}^3$), LP basis
theory guarantees that not more than three positive Farkas multipliers per element and pair are required.

With the $3{\times}3$ ZYX rotation matrix $\mathrm{Rot}(\theta_i, \iota_i, \kappa_i)$, the Platonic-solid packing
problem can be formulated as:
\allowdisplaybreaks
\begin{subequations}
\begin{alignat}{2}
    \min\  R& \label{eq:platonic-obj} && \\[0.3em]
    \text{s.t.}\quad
     R\rho_\ell + \mathbf{n}^k_\ell \cdot \left(\begin{pmatrix}x_i\\y_i\\z_i\end{pmatrix} + \mathrm{Rot}(\theta_i, \iota_i, \kappa_i)\, \mathbf{v}^j_m\right) &\ge 0,
    &&\forall i \in \mathcal{N},\ j \in \mathcal{V}_m, \nonumber \\
    &&&k \in \mathcal{F}_\ell
    \label{eq:platonic-containment} \\[0.3em]
    \begin{pmatrix}a_{i,f}\\b_{i,f}\\c_{i,f}\end{pmatrix} = \mathrm{Rot}(\theta_i, \iota_i, \kappa_i)\, \mathbf{n}^f_m,&
    &&\forall i \in \mathcal{N},\ f \in \mathcal{F}_m
    \label{eq:platonic-normal} \\
     e_{i,f} = a_{i,f}x_i + b_{i,f}y_i + c_{i,f}z_i - \rho_m,&
    &&\forall i \in \mathcal{N},\ f \in \mathcal{F}_m
    \label{eq:platonic-offset} \\[0.3em]
     \sum_{k=1}^{2|\mathcal{F}_m|} \lambda_{i,j,k} &= 1,
    &&\forall i,j \in \mathcal{N},\ i<j
    \label{eq:platonic-farkas-sum} \\
     \sum_{k=0}^{|\mathcal{F}_m|-1} \lambda_{i,j,k+1} \begin{pmatrix}a_{i,k}\\b_{i,k}\\c_{i,k}\end{pmatrix}
      + \sum_{k=0}^{|\mathcal{F}_m|-1} \lambda_{i,j,k+|\mathcal{F}_m|+1} \begin{pmatrix}a_{j,k}\\b_{j,k}\\c_{j,k}\end{pmatrix} &= \mathbf{0},
    &&\forall i,j \in \mathcal{N},\ i<j
    \label{eq:platonic-farkas-xyz} \\
     \sum_{k=0}^{|\mathcal{F}_m|-1} \lambda_{i,j,k+1}\, e_{i,k}
      + \sum_{k=0}^{|\mathcal{F}_m|-1} \lambda_{i,j,k+|\mathcal{F}_m|+1}\, e_{j,k} &\ge 0,
    &&\forall i,j \in \mathcal{N},\ i<j
    \label{eq:platonic-farkas-sep} \\[0.3em]
     0 \le \theta_i, \iota_i, \kappa_i &\le 2\pi,
    &&\forall i \in \mathcal{N}
    \label{eq:platonic-rotation} \\
     \lambda_{i,j,k} &\ge 0,
    &&\forall i,j \in \mathcal{N},\ i<j, \nonumber \\
    &&&1\leq k \leq 2|\mathcal{F}_m| \\
     R &\ge R_{\min}
    \label{eq:platonic-bounds}
\end{alignat}
\end{subequations}
Constraints \eqref{eq:platonic-containment} ensure that all vertices of each inner solid lie within the outer solid by
requiring all oriented vertices of inner solids to satisfy all constraints of the outer solid.

Constraints \eqref{eq:platonic-normal}--\eqref{eq:platonic-offset} define the oriented half-space representation of
each inner solid. For this, each face $f \in \mathcal{F}_m$ of inner solid $i$ is represented by the oriented inward
normal $(a_{i,f}, b_{i,f}, c_{i,f})$ and offset $e_{i,f}$, corresponding to the defining half-space constraint
$a_{i,f}x + b_{i,f}y + c_{i,f}z \ge e_{i,f}$.

Constraints \eqref{eq:platonic-farkas-sum}--\eqref{eq:platonic-farkas-sep} represent the Farkas-based separation
conditions as described above.

Constraints \eqref{eq:platonic-rotation} leave the rotation unrestricted since three-dimensional structures in general
do not have congruence symmetries in each rotational coordinate in contrast to regular polygons in two dimensions.
Finally, the lower bound $R_{\min}$ in \eqref{eq:platonic-bounds} is again derived from the containment requirement
that the volume of the scaled outer $\ell$-solid must be at least the total volume of the $n$ unit inner $m$-solids.

\subsection{Computational Results} 
\label{subsec:platonicresults}

The resulting circumradii for all computed Platonic-solid packing instances are listed in the online supplement\footnote{\url{https://github.com/DominikKamp/Packing}}. Here we list the solutions that are either improving or are new and of an interesting structure.

First, we have found an improving cube-packing solution for the case of packing 11 cubes into a larger cube. The solution is depicted in Figure~\ref{fig:platonic_improving}. This is a well-researched area, so it is surprising that such a solution could be found.
\begin{figure}[ht]
\begin{center}
\includegraphics[width=0.35\linewidth]{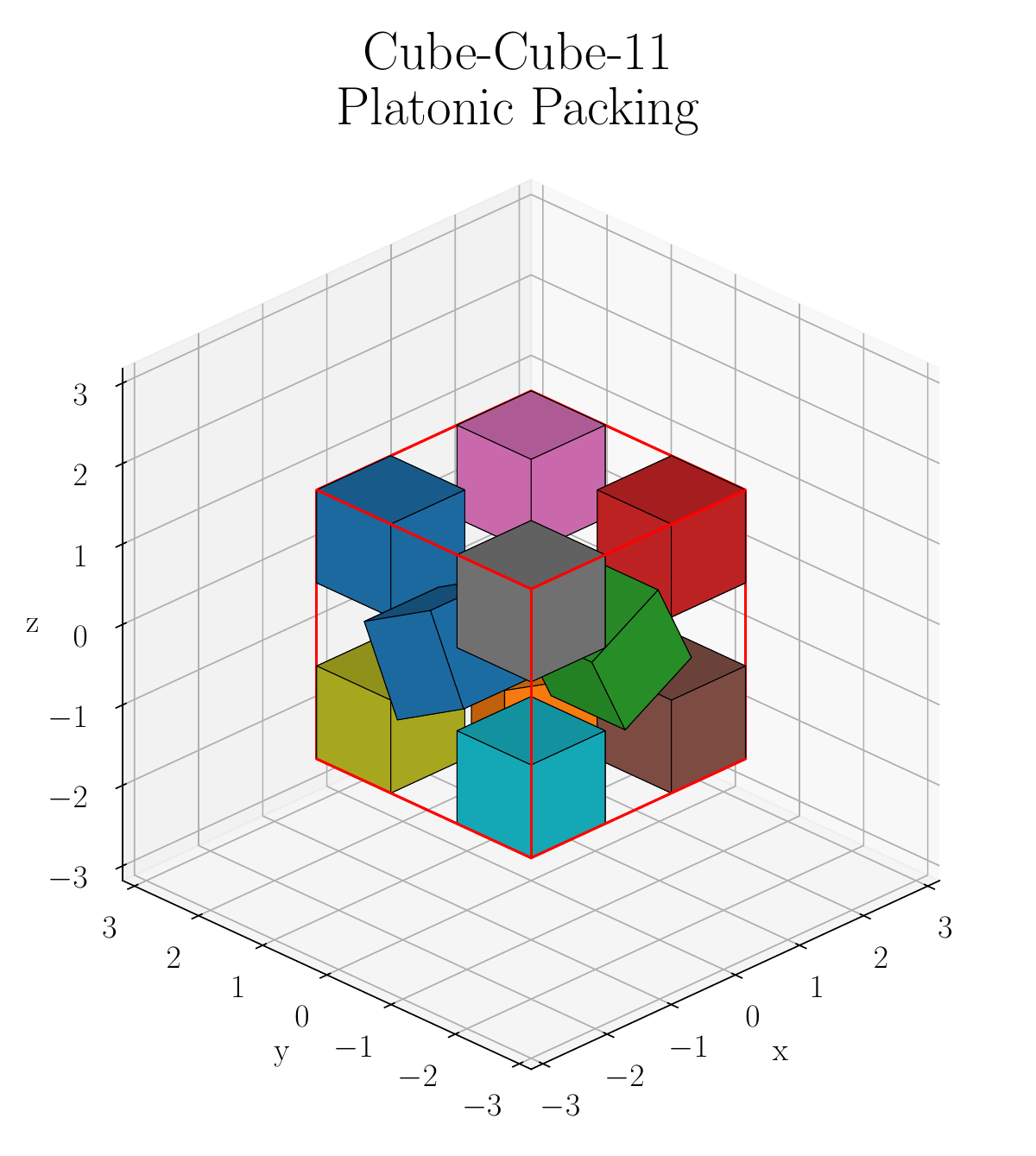}
\end{center}
\caption{Packing of 11 unit-circumradius cubes into a cube with $R = 2.89445$}\label{fig:platonic_improving}
\end{figure}
The numerical details of the solution are summarized in Table~\ref{tab:platonic_improving}.
\begin{table}[ht]
\caption{Improving solution for packing 11 cubes into a cube}\label{tab:platonic_improving}
\begin{tabular}{@{}rrrllr@{}}
\toprule
$\ell$ & $m$ & $n$ & Circumradius & Previous & Source\\
\midrule
3 & 3 & 11 & 2.\textbf{89445} & 2.91210 & Friedman 1998 \\
\bottomrule
\end{tabular}
\end{table}

In addition we have tested all combinations of inner and outer polyhedra with up to 30 inner ones, resulting in $5\cdot5\cdot30=750$ instances. For most of these problems we are the first to report any feasible solutions. A few selected configurations are depicted in Figure~\ref{fig:platonic}. We find it remarkable that such tight configurations could be found with a straightforward model and a general-purpose solver.

\begin{figure}[ht]
\begin{center}
\includegraphics[width=0.30\linewidth]{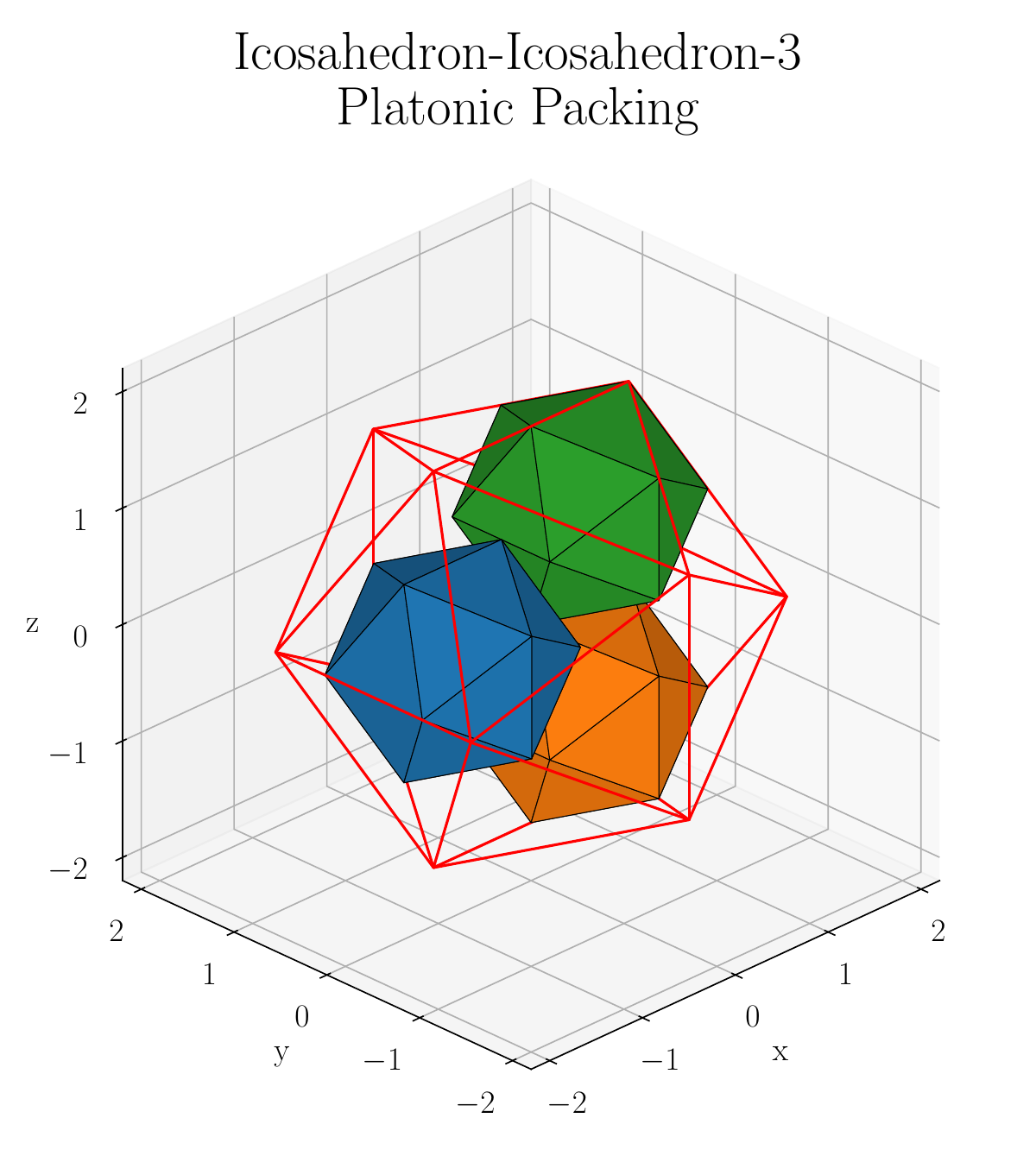}
\includegraphics[width=0.30\linewidth]{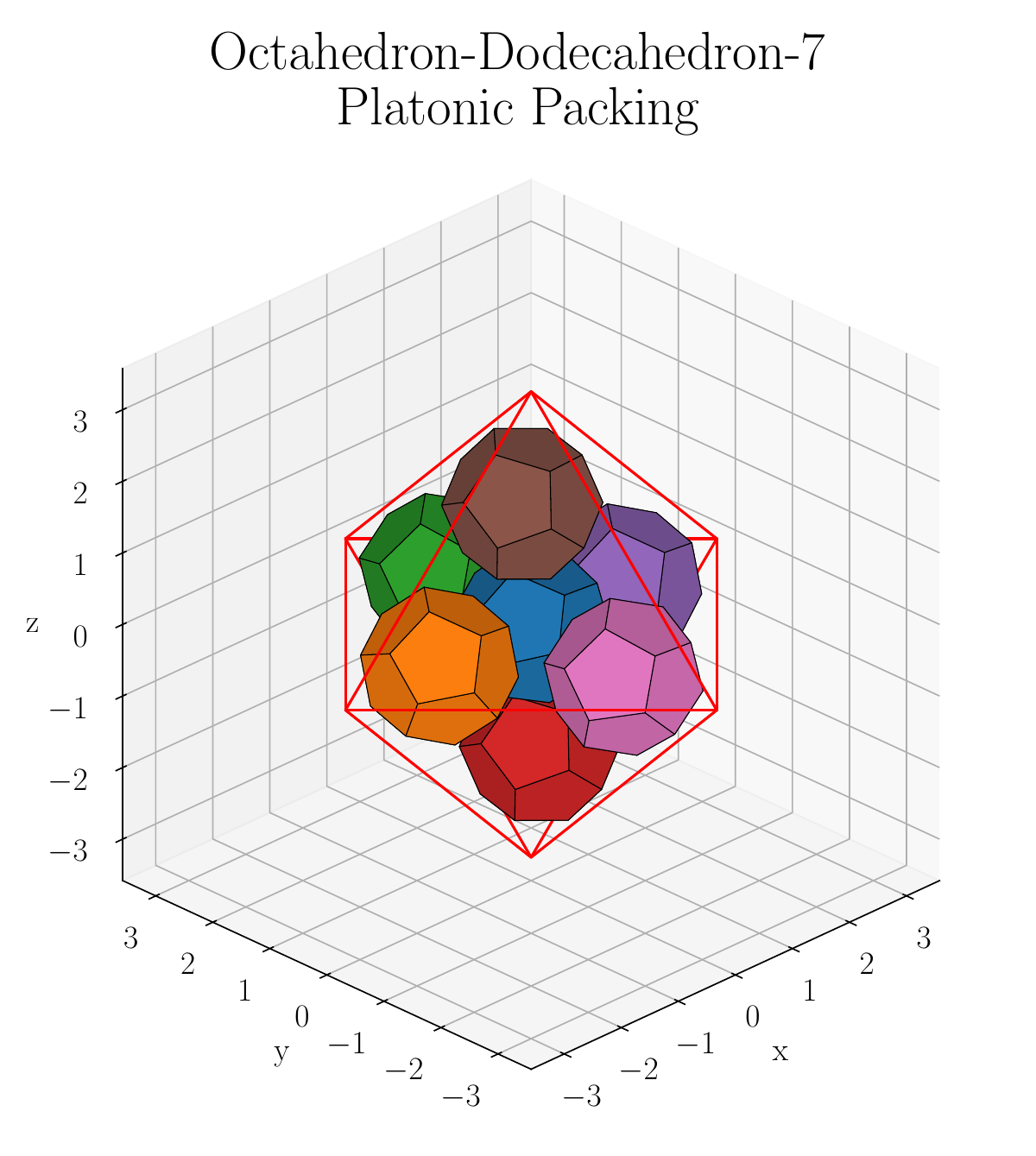}
\includegraphics[width=0.30\linewidth]{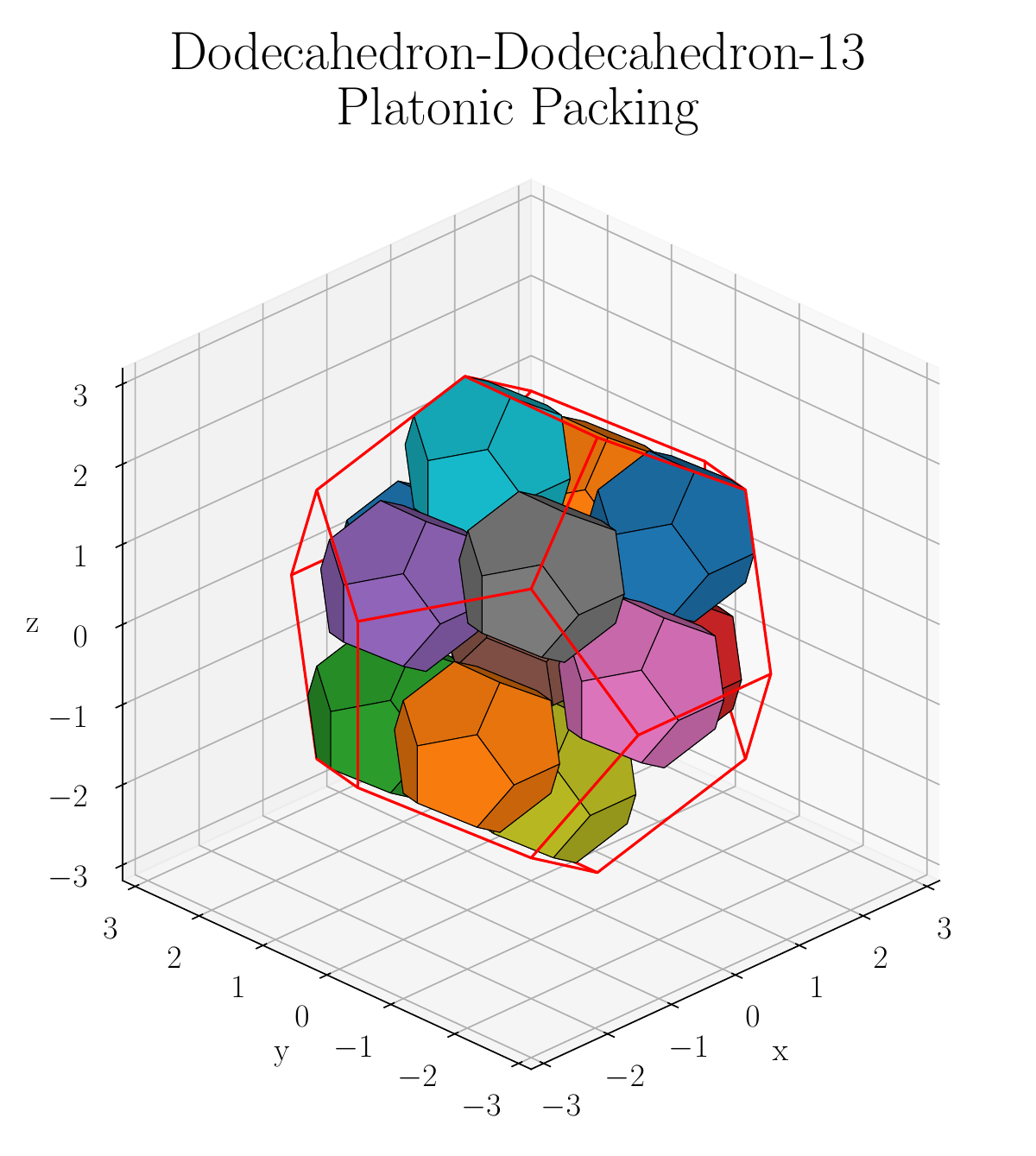}\\
\includegraphics[width=0.30\linewidth]{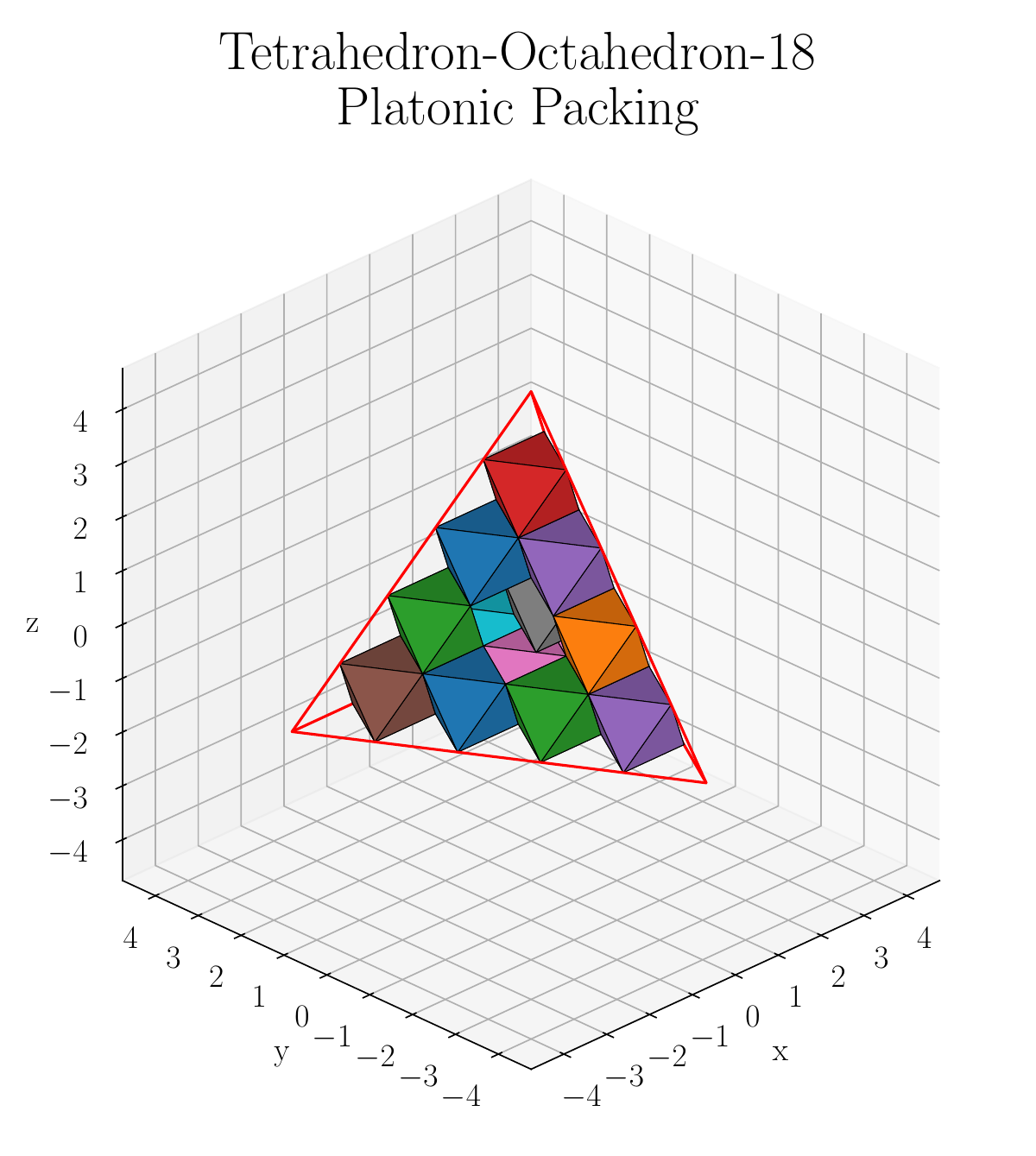}
\includegraphics[width=0.30\linewidth]{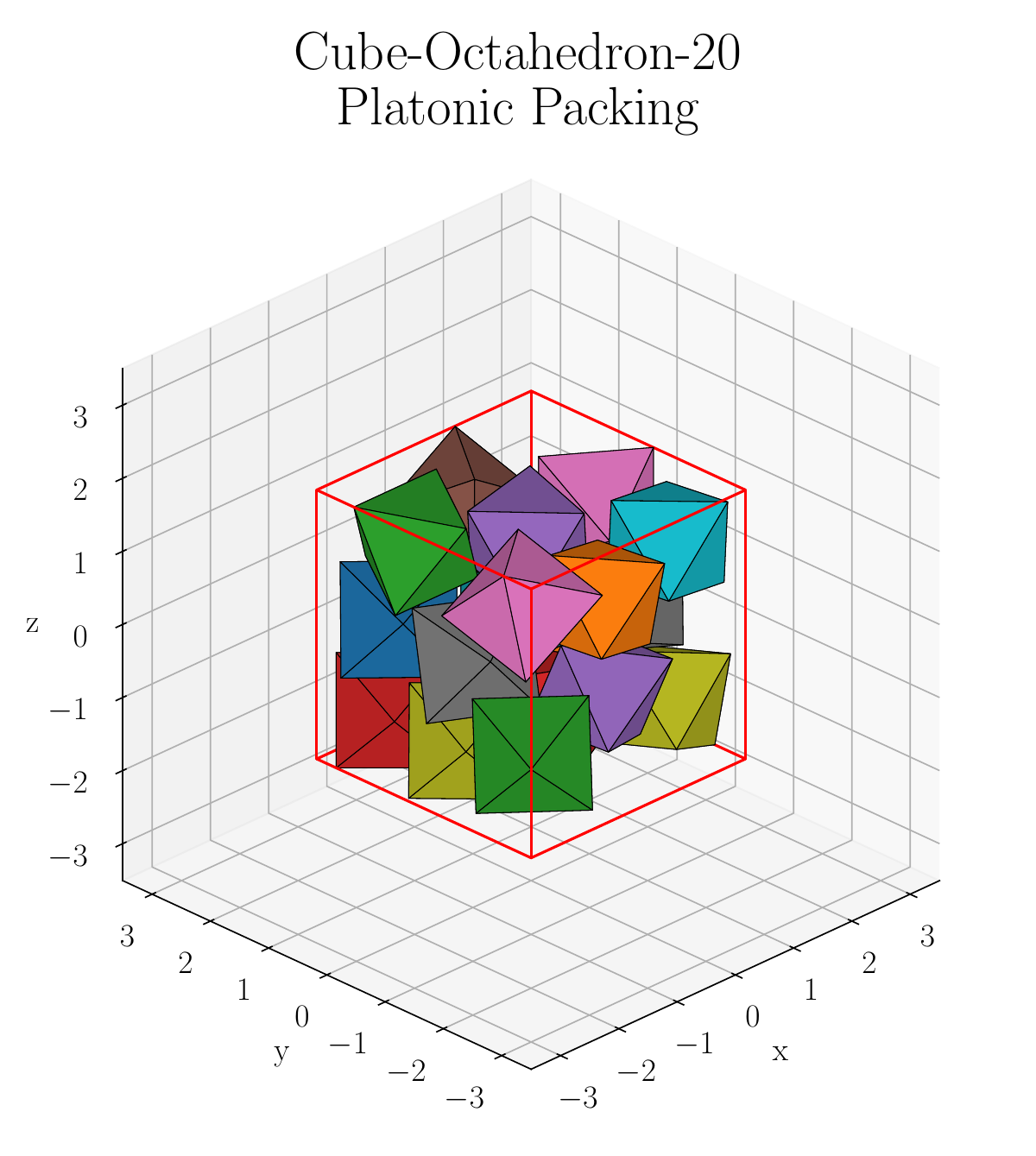}
\includegraphics[width=0.30\linewidth]{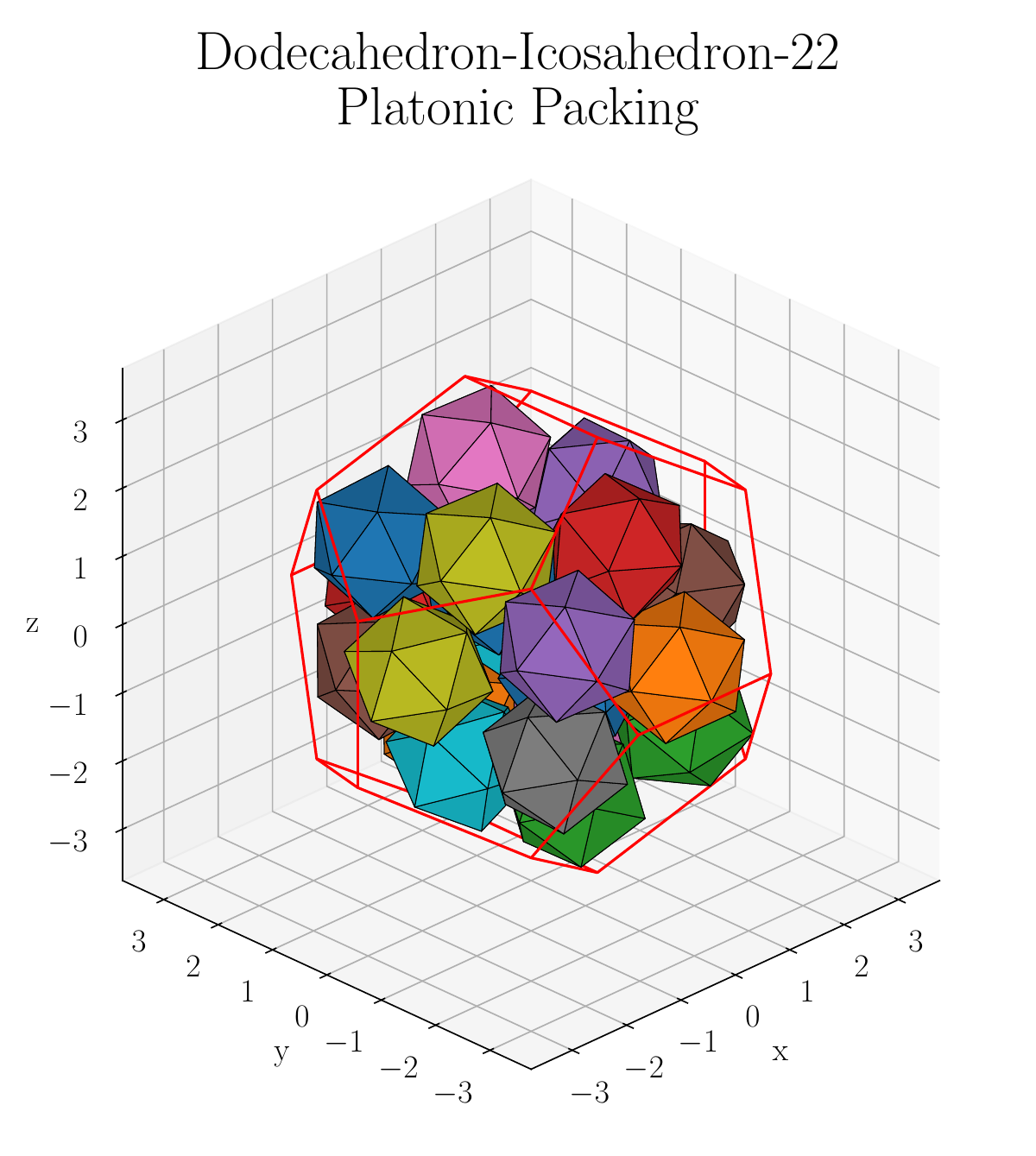}
\end{center}
\caption{Graphical representation of the solutions for selected instances of the Platonic-solid packing problem}\label{fig:platonic}
\end{figure}

The median optimality gaps are significantly larger than for polygons, ranging from $14.5\%$ to $33.3\%$ (average $26.7\%$). The dual bound is improved beyond $R_{\min}$ only for $n \leq 2$, indicating that further model strengthening is necessary to close the gaps.

\subsection{Solution Polishing}

We used a scheme analogous to the polygon-packing study, presented in Section~\ref{sec:polypolish}.

\subsection{Model Variants}\label{subsubsec:platonicvariants}

Similar to polygon packing, we have investigated the performance of the other formulation variants. The same five formulation variants described in Section~\ref{subsubsec:polygonvariants} are tested for Platonic-solid packing, with each modification extended analogously to three dimensions.
\textsc{Dist} adds redundant center-distance constraints $(x_i{-}x_j)^2 + (y_i{-}y_j)^2 + (z_i{-}z_j)^2 \geq (2\rho_m)^2$;
\textsc{Inner} replaces Farkas separation with direct inner separation hyperplanes in~$\mathbb{R}^3$;
\textsc{Farkas} uses individual normalization $\sum_{k} \lambda_{i,j,k+1} \geq 1$ and $\sum_{k} \lambda_{i,j,k+|\mathcal{F}_m|+1} \geq 1$ in place of~\eqref{eq:platonic-farkas-sum};
and \textsc{Sym} extends \textsc{Dist} with $x_i \leq x_{i+1}$ ordering and centroid constraints into the cone spanned by the first half of the first angular sector of the first face. Table~\ref{tab:platonic_overall} reports the same metrics for the Platonic-solid packing instances.
\begin{longtable}{rrrrrr}
  \caption{Platonic-solid packing: overall formulation comparison across 750 instances. Solutions found, median and mean relative difference to \textsc{Dist}, and diff counts per formulation.} \label{tab:platonic_overall} \\
  \toprule
  metric & \textsc{Dist} & \textsc{Nodist} & \textsc{Inner} & \textsc{Farkas} & \textsc{Sym} \\
  \midrule
  \endfirsthead
  \multicolumn{6}{l}{\tablename\ \thetable{} -- \textit{continued}} \\
  \toprule
  metric & \textsc{Dist} & \textsc{Nodist} & \textsc{Inner} & \textsc{Farkas} & \textsc{Sym} \\
  \midrule
  \endhead
  \endfoot
  \bottomrule
  \endlastfoot
  found & 429 & 263 & 256 & 208 & 417 \\
  median & +0.000\% & +0.000\% & +2.610\% & +0.000\% & +0.763\% \\
  mean & +0.000\% & +0.836\% & +34.049\% & +1.263\% & +1.948\% \\
  better & 0 & 18 & 2 & 7 & 18 \\
  worse & 0 & 108 & 160 & 84 & 264 \\
\end{longtable}
The overall pattern resembles the polygon case, while reflecting the increased difficulty of the three-dimensional problem.
\textsc{Dist} again serves as reference.
Notably, both \textsc{Nodist} and \textsc{Sym} produce~18 instances that are strictly better than \textsc{Dist}.
For Platonic-solid packing, \textsc{Inner} faces the largest quality degradation ($+2.610\%$ median, $+34.0\%$ mean) and~160 strictly worse instances.
\textsc{Farkas} matches \textsc{Dist} in median quality but is strictly worse on~84 instances versus~7 where it is better.

This is a whole new and wide-open area of research. We expect that some of the solutions we have found will soon be improved either by us or by other teams.

\section{Solver Performance and Timings}\label{sec:solvingperformance}

Table~\ref{tab:solving_overview} summarizes the solving behavior over the instances with a known reference, reporting the geometric mean of the time to reach an incumbent within $1\%$, $5\%$ and $10\%$ of it.
The reported times count the global solve only, excluding the preceding multistart ramp-up of the Xpress runs.
An incumbent within $1\%$ of reference is found within seconds for the circle packings, and within $5\%$ of reference within minutes for polygon and Platonic-solid packing.

\begin{table}[htbp]
  \centering
  \caption{Solution time overview over the referenced instances of each problem ($10\,000\,$s time limit). 
 The \emph{time} columns give the geometric mean shifted by one second of the time in seconds to reach an incumbent within $1\%$, $5\%$ and $10\%$ of the reference value, when one is available. Polygon and Platonic use their distance formulations.}
  \label{tab:solving_overview}
\begin{tabular}{lr rrr}
  \toprule
  & & \multicolumn{3}{c}{time} \\
  \cmidrule(lr){3-5}
  Problem & $N$ & $1\%$ & $5\%$ & $10\%$ \\
  \midrule
  Circle packing (square) & 32 & 3.1 & 2.1 & 2.1 \\
  Circle packing (rectangle) & 30 & 5.8 & 3.1 & 2.9 \\
  Polygon packing & 365 & 1074.7 & 302.0 & 161.3 \\
  Platonic-solid packing & 30 & 482.2 & 360.3 & 310.9 \\
  \bottomrule
  \end{tabular}
\end{table}
For the circles-in-ellipse problem, the best-known solution was obtained within a minute with a simple solver run. That problem did not require the elaborate computational setup of the other problems in this paper.

\section{Outlook and Conclusion}
\label{thats_it_folks}

In this paper we revisited a collection of geometric packing problems for which the community has been reporting best-known configurations over decades, and we obtained new best solutions for a substantial subset of them. In particular, we have improved best-known results for circle packing with variable radii in fixed-perimeter rectangles, for packing unit circles into minimum-area ellipses, for a broad range of regular polygon-packing instances, and even for the well-studied case of packing $11$ cubes into a cube. For the polygon and the ellipse-packing problems, we introduced new modeling methodologies; in particular, a practical use case for the S-lemma.

Beyond improving existing benchmarks, we also argued that some of the modeling ideas we used generalize naturally to higher-dimensional settings. In particular, our Farkas-based non-overlap formulation for polygon packing generalizes to polyhedron packing in higher dimensions, which motivated us to introduce the problem class of packing Platonic solids into Platonic solids, and to report high-quality solutions for a first collection of instances which we derived from our models. 

The central message, however, is not about packing. Rather, this work demonstrates that global nonlinear optimization has become an out-of-the-box technology. The improvements reported here were obtained with standard, general-purpose solvers (FICO Xpress and SCIP) on compact NLP formulations, without any problem-specific solution algorithm or heuristic: the expertise went into the models, the search was left to the solvers. Some of the authors have already argued about the importance of this point in~\cite{orms26}.

In this sense, modern global solvers provide a competitive baseline for optimization problems involving nonlinearities, including the nonconvex geometric problems that have recently attracted attention through LLM-driven discovery. More than that, global optimization offers complementary strengths: it operates on transparent models and provides dual bounds, it generalizes when the problem to be solved changes slightly, and it does not require an expensive training/learning phase.

From a modeling perspective, we unsurprisingly found that formulation choices matter.
Our systematic comparisons of five formulation variants for polygon and Platonic-solid packing (Sections~\ref{subsubsec:polygonvariants} and ~\ref{subsubsec:platonicvariants}) confirm the value of standard strengthening techniques.
The addition of redundant quadratic distance constraints (\textsc{Dist}) emerged as the single most impactful modification, substantially increasing solution quality.
These constraints are primarily thought to strengthen the dual bound, but we observed that this consequently also leads to stronger primal solutions, a behavior known from MIP~\cite{Berthold2006}.
Symmetry-breaking constraints (\textsc{Sym}) degraded solution quality.
Alternative Farkas normalizations (\textsc{Farkas}) and direct inner separation (\textsc{Inner}) did not improve upon the base formulation.

There are also clear caveats. The most immediate one is the quadratic growth in model size induced by pairwise non-overlap constraints. This becomes a computational obstacle when one moves towards instances with hundreds or thousands of objects. One simple way to mitigate that is to couple the non-overlap constraints with the symmetry-breaking constraints that sort the objects left to right, and only add the non-overlap constraints if the objects are not too far from each other. A more complicated but probably more efficient method is to  separate violated non-overlap constraints and other valid inequalities dynamically. Modern global solvers are already prepared for this line of research: Xpress offers callback mechanisms for branching, cutting, and various other solver interactions, and the plugin-based design of SCIP provides analogous extensibility.

The polygon and Platonic-solid packing formulations could be further strengthened by:
\begin{enumerate}
\item  simplifying the separation constraints combinatorially by indicating the separating vertices through integer variables, as suggested by the sparsity result discussed in connection with Lemma~\ref{lem:polygon-separation};
\item exploiting the structure of the separation constraints in a cutting-plane framework to dynamically add violated constraints only for pairs that are close to each other; 
\item using density propagation, deriving lower bounds on~$R$ from area conditions on restricted element coordinates during spatial branch-and-bound.
\end{enumerate}

Finally, while AlphaEvolve was a strong trigger to look at classical geometric benchmarks from a model-and-solve perspective, not every problem from that line of research~\cite{novikov2025alphaevolve,georgiev2025mathematical} is equally amenable to compact NLP modeling. The kissing number problem, for example, is conceptually simple to model but immediately suffers from the same combinatorial growth discussed above. 

On the other hand, our earlier work~\cite{berthold2026global} already demonstrated that global NLP solvers can be highly effective on distance-geometry-style instances, and we are optimistic that further problems such as the Heilbronn problem (for first results on different variants, see~\cite{monji2025solving,sudermannmerx2026}) can be attacked in the same way.

Looking ahead, we expect the scope of problems, both academic research and industrial applications, for which a model-and-solve approach is competitive to expand rapidly with continued progress in global optimization technology. The mixed-integer linear programming success story suggests a plausible trajectory: with the right combination of general yet structure-aware cuts, presolving, propagation, and primal heuristics, computational MINLP can become a black-box tool that is routinely usable at scale. We should stop looking at \emph{global} MINLP solvers as generalizations of \emph{local} nonlinear solvers, and instead view them as extensions of linear MIP solvers: they offer the same guarantees (primal solution, dual bound), but can handle any kind of nonlinearity. This opens the door to a much larger user base.

The recipe is, in principle, known: step back, mine the already existing rich literature on algorithms and theory for nonlinear optimization, and turn that research into robust, general-purpose solver technology. The packing case studies in this paper illustrate both that this direction is already paying off and that there remains substantial research potential. The beauty of research on general-purpose global optimization algorithms is its impact: improvement transfers across models and domains and can lead to outsized real-world impact.

\bmhead{Acknowledgements}

We are grateful to Liding Xu (ZIB) and the FICO Xpress team, in particular Bruno Vieira and Susanne Heipcke, for their contributions during an early phase of this project.
We thank Anita Hovanesian (ZIB) for her support with visualizing solutions.
We are indebted to Erich Friedman for his great collection of packing problems and benchmark solutions, which sparked this work, and for incorporating our new findings on his page.

\bmhead{Funding disclosure}
Research reported in this paper was partially supported through the Research Campus MODAL funded by the German Federal Ministry of Education and Research (fund numbers 05M14ZAM, 05M20ZBM) and the Deutsche Forschungsgemeinschaft (DFG) through the DFG Cluster of Excellence MATH+.

\bmhead{Competing interests} Timo Berthold and Imre Pólik are employed by Fair Isaac Corporation (FICO), which develops and sells FICO Xpress, one of the two solvers used in this study. Dominik Kamp, Gioni Mexi and Sebastian Pokutta are affiliated with Zuse Institute Berlin, which develops the SCIP Optimization Suite, the other solver used.
\bibliography{bibliography}

\backmatter

\end{document}